\documentclass[11pt]{article}

\input xy

\xyoption{all}

\usepackage{amssymb,amsbsy,amsthm,amsmath,mathrsfs,graphicx}
\usepackage{times}

\newtheorem{thm}{Theorem}[section]
\newtheorem{lem}[thm]{Lemma}
\newtheorem{prop}[thm]{Proposition}
\newtheorem{cor}[thm]{Corollary}

\newtheorem{dfn}[thm]{Definition}

\newtheorem*{thm*}{Theorem}
\newtheorem*{cor*}{Corollary}

\theoremstyle{remark}

\newtheorem*{ex}{Example}
\newtheorem*{rmk}{Remark}

\renewcommand{\bf}[1]{\mathbf{#1}}
\renewcommand{\rm}[1]{\mathrm{#1}}
\renewcommand{\cal}[1]{\mathcal{#1}}

\newcommand{\bbN}{\mathbb{N}}

\newcommand{\bbZ}{\mathbb{Z}}

\newcommand{\bfV}{\mathbf{V}}

\newcommand{\bfX}{\mathbf{X}}
\newcommand{\bfY}{\mathbf{Y}}
\newcommand{\bfZ}{\mathbf{Z}}

\newcommand{\sfE}{\mathsf{E}}
\newcommand{\sfF}{\mathsf{F}}
\newcommand{\sfG}{\mathsf{G}}
\newcommand{\sfV}{\mathsf{V}}

\renewcommand{\d}{\mathrm{d}}

\newcommand{\C}{\mathcal{C}}

\newcommand{\I}{\mathcal{I}}
\newcommand{\J}{\mathcal{J}}
\newcommand{\K}{\mathcal{K}}

\renewcommand{\P}{\mathcal{P}}

\newcommand{\frH}{\mathfrak{H}}

\newcommand{\frK}{\mathfrak{K}}
\newcommand{\frL}{\mathfrak{L}}

\renewcommand{\L}{\Lambda}

\renewcommand{\S}{\Sigma}

\renewcommand{\a}{\alpha}
\renewcommand{\b}{\beta}

\renewcommand{\l}{\lambda}
\newcommand{\s}{\sigma}
\renewcommand{\phi}{\varphi}

\renewcommand{\hat}[1]{\widehat{#1}}
\newcommand{\ol}[1]{\overline{#1}}

\newcommand{\fin}{\nolinebreak\hspace{\stretch{1}}$\lhd$}

\renewcommand{\t}[1]{\widetilde{#1}}
\renewcommand{\to}{\longrightarrow}

\newcommand{\actson}{\curvearrowright}

\newcommand{\uhr}{\!\upharpoonright}

\begin{document}

% "Title of the paper"
\title{\Large{NON-CONVENTIONAL ERGODIC AVERAGES FOR SEVERAL COMMUTING ACTIONS OF AN AMENABLE GROUP}}

\author{\textsc{Tim Austin}\footnote{Research supported by a fellowship from the Clay Mathematics Institute}\\ \\ \small{Courant Institute, New York University}\\ \small{New York, NY 10012, U.S.A.}\\ \small{\texttt{tim@cims.nyu.edu}}}

\date{}

\maketitle

% indicate corresponding author with \corref{}
% \author{\fnms{John} \snm{Smith}\corref{}\ead[label=e1]{smith@foo.com}\thanksref{t1}}
% \thankstext{t1}{Thanks to somebody} 
% \address{line 1\\ line 2\\ printead{e1}}
% \affiliation{Some University}

%\author{\fnms{Tim} \snm{Austin}\ead[label=e1]{tim@cims.nyu.edu}}
%\address{}
%\affiliation{Courant Institute, New York University}
%\and
%\author{\fnms{???} \snm{???}\ead[label=e2]{???}}
%\address{\printead{e2}}
%\affiliation{???}

%\runauthor{Tim Austin}

\begin{abstract}
Let $(X,\mu)$ be a probability space, $G$ a countable amenable group and $(F_n)_n$ a left F\o lner sequence in $G$.
This paper analyzes the non-conventional ergodic averages
\[\frac{1}{|F_n|}\sum_{g \in F_n}\prod_{i=1}^d (f_i\circ T_1^g\cdots T_i^g)\]
associated to a commuting tuple of $\mu$-preserving actions $T_1$, \ldots, $T_d:G\actson X$ and $f_1$, \ldots, $f_d \in L^\infty(\mu)$.   We prove that these averages always converge in $\|\cdot\|_2$, and that they witness a multiple recurrence phenomenon when $f_1 = \ldots = f_d = 1_A$ for a non-negligible set $A\subseteq X$.  This proves a conjecture of Bergelson, McCutcheon and Zhang.  The proof relies on an adaptation from earlier works of the machinery of sated extensions.
\end{abstract}

\tableofcontents

\section{Introduction}

Let $(X,\mu)$ be a probability space, $G$ a countable amenable group, and $T_1$, \ldots, $T_d:G\actson (X,\mu)$ a tuple of $\mu$-preserving actions of $G$ which commute, meaning that
\[i\neq j \quad \Longrightarrow \quad T_i^gT_j^h = T_j^hT_i^g \quad \forall g,h \in G.\]
Also, let $(F_n)_n$ be a left F\o lner sequence of subsets of $G$; this will be fixed for the rest of the paper.

In this context, Bergelson, McCutcheon and Zhang have proposed in~\cite{BerMcCZha97} the study of the non-conventional ergodic averages
\begin{eqnarray}\label{eq:Lambda}
\L_n(f_1,\ldots,f_d) := \frac{1}{|F_n|}\sum_{g \in F_n}\prod_{i=1}^d(f_i\circ T_1^g\cdots T_i^g)
\end{eqnarray}
for functions $f_1,\ldots,f_d \in L^\infty(\mu)$.  These are an analog for commuting $G$-actions of the non-conventional averages for a commuting tuple of transformations, as introduced by Furstenberg and Katznelson~\cite{FurKat78} for their proof of the multi-dimensional generalization of Szemer\'edi's Theorem.  Other analogs are possible, but the averages above seem to show the most promise for building a theory: this is discussed in~\cite{BerMcCZha97} and, for topological dynamics, in~\cite{BerHin92}, where some relevant counterexamples are presented.

The main results of~\cite{BerMcCZha97} are that these averages converge, and that one has an associated multiple recurrence phenomenon, when $d=2$.  The first of these conclusions can be extended to arbitrary $d$ along the lines of Walsh's recent proof of convergence for polynomial nilpotent non-conventional averages~(\cite{Walsh12}).

\vspace{7pt}

\noindent\textbf{Theorem A.}\quad \emph{In the setting above, the functional averages $\L_n(f_1,\ldots,f_d)$ converge in the norm of $L^2(\mu)$ for all $f_1$, \ldots, $f_d \in L^\infty(\mu)$.}

\vspace{7pt}

Zorin-Kranich has made the necessary extensions to Walsh's argument in~\cite{Zor13}.  However, that proof gives essentially no information about the limiting function, and in particular does not seem to enable a proof of multiple recurrence.  The present paper gives both a new proof of Theorem A, and a proof of the following.

\vspace{7pt}

\noindent\textbf{Theorem B.}\quad \emph{If $\mu(A) > 0$, then
\begin{multline*}
\lim_{n\to\infty}\int_X \L_n(1_A,\ldots,1_A)\,\d\mu\\ = \lim_{n\to\infty}\frac{1}{|F_n|}\sum_{g \in F_n}\mu\big(T_1^{g^{-1}}A\cap \cdots \cap (T_1^{g^{-1}}\cdots T_d^{g^{-1}})A\big) > 0.
\end{multline*}
In particular, the set
\[\big\{g \in G\,\big|\ \mu\big(T_1^{g^{-1}}A\cap \cdots \cap (T_1^{g^{-1}}\cdots T_d^{g^{-1}})A\big) > 0\big\}\]
has positive upper Banach density relative to $(F_n)_{n\geq 1}$.}

\vspace{7pt}

As in the classical case of~\cite{FurKat78}, this implies the following Szemer\'edi-type result for amenable groups.

\begin{cor*}
If $E \subseteq G^d$ has positive upper Banach density relative to  $(F_n^d)_{n\geq 1}$, then the set
\begin{multline*}
\big\{g \in G\,\big|\ \exists (x_1,\ldots,x_d) \in G^d\\ \hbox{s.t.}\ \{(g^{-1}x_1,x_2,\ldots,x_d),\ldots,(g^{-1}x_1,\ldots,g^{-1}x_d)\} \subseteq E\big\}
\end{multline*}
has positive upper Banach density relative to $(F_n)_{n\geq 1}$. \qed
\end{cor*}

This deduction is quite standard, and can be found in~\cite{BerMcCZha97}.

Our proofs of Theorems A and B are descended from some work for commuting tuples of transformations: the proof of non-conventional-average convergence in~\cite{Aus--nonconv}, and that of multiple recurrence in~\cite{Aus--newmultiSzem}.  Both of those papers offered alternatives to earlier proofs, using new machinery for extending an initially-given probability-preserving action to another action under which the averages behave more simply.  The present paper will adapt to commuting tuples of $G$-actions the notion of a `sated extension', which forms the heart of the streamlined presentation of that machinery in~\cite{Aus--thesis}.  Further discussion of this method may be found in that reference.

The generalization of the notion of satedness is nontrivial, but fairly straightforward: see Section~\ref{sec:func} below.  However, more serious difficulties appear in how it is applied.  Heuristically, if a given system satisfies a satedness assumption, then, in any extension of that system, this constrains how some canonical $\s$-subalgebra `sits' relative to the $\s$-algebra lifted from the original system.  An appeal to satedness always relies on constructing a particular extension for which this constraint implies some other desired consequence.  The specific constructions of system extensions used in~\cite{Aus--nonconv,Aus--newmultiSzem,Aus--thesis} do not generalize to commuting actions of a non-Abelian group $G$.  This is because they rely on the commutativity of the diagonal actions $T_i\times \cdots \times T_i$ of $G$ on $X^d$ with the `off-diagonal' action generated by
\[T_1^g\times \cdots \times (T_1^g\cdots T^g_d), \quad g \in G.\]

Thus, a key part of this paper is a new method of extending probability-preserving $G^d$-systems.  It is based on a version of the Host-Kra self-joinings from~\cite{HosKra05} and~\cite{Hos09}.  It also relies on a quite general result about probability-preserving systems, which may be of independent interest: Theorem~\ref{thm:recoverG} asserts that, given a probability-preserving action of a countable group and an extension of that action restricted to a subgroup, a compatible further extension may be found for the action of the whole group.

Developing ideas from~\cite{HosKra05}, we find that the asymptotic behaviour of our non-conventional averages can be estimated by certain integrals over these Host-Kra-like extensions (Theorem~\ref{thm:ineq}).  On the other hand, a suitable satedness assumption on a system gives extra information on the structure of those extensions, and combining these facts then implies simplified behaviour for the non-conventional averages for that system.  Finally, the existence of sated extensions for all systems (Theorem~\ref{thm:sateds-exist}) then enables proofs of convergence and multiple recurrence similar to those in~\cite{Aus--nonconv} and~\cite{Aus--newmultiSzem}, respectively.

An interesting direction for further research is suggested by the work of Bergelson and McCutcheon in~\cite{BerMcC07}.  They study multiple recurrence phenomena similar to Theorem B when $d=3$, but without assuming that the group $G$ is amenable.  As output, they prove that the set
\[\big\{g \in G\,\big|\ \mu\big(T_1^{g^{-1}}A\cap (T_1^{g^{-1}}T_2^{g^{-1}} )A \cap (T_1^{g^{-1}}T_2^{g^{-1}} T_3^{g^{-1}})A\big) > 0\big\}\]
is `large' in a sense adapted to non-amenable groups, in terms of certain special ultrafilters in the Stone-\v{C}ech compactification of $G$.  In particular, their result implies that this set is syndetic in $G$.  Could their methods be combined with those below to extend this result to larger values of $d$?

\section{Generalities on actions and extensions}

\subsection{Preliminaries}

If $d \in \bbN$ then $[d] := \{1,2,\ldots,d\}$, and more generally if $a,b \in \bbZ$ with $a \leq b$ then
\[(a;b] = [a+1;b] = [a+1;b+1) = (a;b+1):= \{a+1,\ldots,b\}.\]
The power set of $[d]$ will be denoted $\cal{P}[d]$, and we let $\binom{[d]}{\geq p}:= \{e \in \cal{P}[d]\,|\ |e|\geq p\}$.

Next, if $\cal{A} \subseteq \P[d]$, then it is an \textbf{up-set} if
\[a,b \in \cal{A} \quad \Longrightarrow \quad a\cup b \in\cal{A}.\]
The set
\[\langle e \rangle := \{a \subseteq [d]\,|\ a \supseteq e\}\]
is an up-set for every $e \subseteq [d]$, and every up-set is a union of such examples. On the other hand, $\cal{B} \subseteq \P[d]$ is an \textbf{antichain} if
\[a,b\in \cal{B} \quad \hbox{and} \quad a \subseteq b \quad \Longrightarrow \quad a = b.\]
Any up-set contains a unique anti-chain of minimal elements.

Standard notions from probability theory will be assumed throughout this paper.  If $(X,\mu)$ is a probability space with $\s$-algebra $\S$, and if $\Phi,\S_1,\S_2 \subseteq \S$ are $\s$-subalgebras with $\Phi \subseteq \S_1\cap \S_2$, then $\S_1$ and $\S_2$ are \textbf{relatively independent} over $\Phi$ under $\mu$ if
\[\int_X fg\,\d\mu = \int_X\sfE_\mu(f\,|\,\Phi)\sfE_\mu(g\,|\,\Phi)\,\d\mu\]
whenever $f,g \in L^\infty(\mu)$ are $\S_1$- and $\S_2$-measurable, respectively.  Relatedly, if $(X,\mu)$ is standard Borel, then on $X^2$ we may form the \textbf{relative product} measure $\mu\otimes_\Phi \mu$ over $\Phi$ by letting $x\mapsto \mu_x$ be a disintegration of $\mu$ over the $\s$-subalgebra $\Phi$ and then setting
\[\mu\otimes_\Phi\mu := \int_X \mu_x\otimes \mu_x\,\mu(\d x).\]

If $G$ is a countable group, then a \textbf{$G$-space} is a triple $(X,\mu,T)$ consisting of a probability space $(X,\mu)$ and an action $T:G\actson X$ by measurable, $\mu$-preserving transformations.  Passing to an isomorphic model if necessary, we will henceforth assume that $(X,\mu)$ is standard Borel.  Often, a $G$-space will also be denoted by a boldface letter such as $\bfX$.

If $\bfX = (X,\mu,T)$ is a $G$-space, then $\S_X$ or $\S_\bfX$ will denote its $\s$-algebra of $\mu$-measurable sets.  A \textbf{factor} of such a $G$-space is a $\s$-subalgebra $\Phi \leq \S_\bfX$ which is globally $T$-invariant, meaning that
\[A \in \Phi \quad \Longleftrightarrow \quad T^g(A) \in \Phi \quad \forall g \in G.\]
Relatedly, a \textbf{factor map} from one $G$-space $\bfX = (X,\mu,T)$ to another $\bfY = (Y,\nu,S)$ is a measurable map $\pi:X\to Y$ such that $\pi_\ast\mu = \nu$ and $S^g\circ \pi = \pi \circ T^g$ for all $g \in G$, $\mu$-a.e.  In this case, $\pi^{-1}(\S_\bfY)$ is a factor of $\bfX$.  Such a factor map is also referred to as a \textbf{$G$-extension}, and $\bfX$ may be referred to as an \textbf{extension} of $\bfY$.

On the other hand, if $\bfX = (X,\mu,T)$ is a $G$-space and $H \leq G$, then the \textbf{$H$-subaction} of $\bfX$, denoted $\bfX^{\uhr H} = (X,\mu,T^{\uhr H})$, is the $H$-space with probability space $(X,\mu)$ and action given by the transformations $(T^h)_{h \in H}$.  The associated $\s$-algebra of $H$-almost-invariant sets,
\[\{A \in \S_X\,|\ \mu(T^h(A)\triangle A) = 0\ \forall h \in H\},\]
will be denoted by either $\S_\bfX^H$ or $\S_\bfX^{T^{\uhr H}}$, as seems appropriate.

In the sequel, we will often consider a space $(X,\mu)$ endowed with a commuting tuple $T_1$, \ldots, $T_d$ of $G$-actions.  Slightly abusively, we shall simply refer to this as a `$G^d$-action' or `$G^d$-space' (leaving the distinguished $G$-subactions to the reader's understanding), and denote it by $(X,\mu,T_1,\ldots,T_d)$.  Also, if $(X,\mu,T_1,\ldots,T_d)$ is a $G^d$-space and $a,b \in [d]$ with $a \leq b$, then we shall frequently use the notation
\[T_{[a;b]}^g = T_{(a-1;b]}^g = T_{[a,;b+1)}^g := T_a^gT_{a+1}^g \cdots T_b^g \quad \forall g \in G.\]
Because the actions $T_i$ commute, this defines another $G$-action for each $a,b$.

\subsection{Actions of groups and their subgroups}

Our approach to Theorems A and B is descended from the notions of `pleasant' and `isotropized' extensions.  These were introduced in~\cite{Aus--nonconv} and~\cite{Aus--newmultiSzem} respectively, where they were used to give new proofs of the analogs of Theorems A and B for commuting tuples of single transformations.

Subsequently, the more general notion of `sated' extensions was introduced in~\cite{Aus--thesis}.  It simplifies and clarifies those earlier ideas as special cases.  In this paper we shall show how `sated' extensions can be adapted to the non-Abelian setting of Theorems A and B.

An important new difficulty is that we will need to consider certain natural $\s$-subalgebras of a probability-preserving $G$-spaces which need not be factors in case $G$ is not Abelian.  This subsection focuses on a key tool for handling this situation, which seems to be of interest in its own right. Given $H \leq G$, it enables one to turn an extension of an $H$-subaction into an extension of a whole $G$-action.  Satedness will then be introduced in the next subsection.

\begin{thm}\label{thm:recoverG}
Suppose $H \leq G$ is an inclusion of countable groups, that $\bfX = (X,\mu,T)$ is a $G$-space and that
\[\bfY = (Y,\nu,S) \stackrel{\b}{\to} \bfX^{\uhr H}\]
is an extension of $H$-spaces.  Then there is an extension of $G$-spaces $\t{\bfX}\stackrel{\pi}{\to} \bfX$ which admits a commutative diagram of $H$-spaces
\begin{center}
$\phantom{i}$\xymatrix{
\t{\bfX}^{\uhr H} \ar^{\pi}[rr]\ar_\a[dr] && \bfX^{\uhr H}\\
& \bfY\ar_\b[ur].
}
\end{center}
\end{thm}

This theorem was proved for Abelian $G$ and $H$ in~\cite[Subsection 3.2]{Aus--lindeppleasant1}.  The non-Abelian case is fairly similar.

\begin{proof}
We shall construct the new $G$-space $\t{\bfX}$ by a kind of `relativized' co-induction of $\bfY$ over $\bfX$, and then show that it has the necessary properties.

The construction of a suitable standard Borel dynamical system $(\t{X},\t{T})$, deferring the construction of the measure, is easy. Let
\[\t{X} := \{(y_g)_g \in Y^G\,|\ y_{gh} = S^{h^{-1}}y_g\ \hbox{and}\ \b(y_g) = T^{g^{-1}}\b(y_e)\ \forall g\in G, h \in H\},\]
and let $\t{T}:G\actson \t{X}$ be the restriction to $\t{X}$ of the left-regular representation:
\[\t{T}^k((y_g)_{g \in G}) = (y_{k^{-1}g})_{g\in G}\]
(it is easily seen that this preserves $\t{X} \subseteq Y^G$).

Also, let
\[\a:\t{X}\to Y: (y_g)_g\mapsto y_e\]
and
\[\pi:= \b\circ \a:\t{X}\to X: (y_g)_g \mapsto \b(y_e).\]

These maps fit into a commutative diagram of the desired shape by construction.  It remains to specify a suitable measure $\t{\mu}$ on $\t{X}$.  It will be constructed as a measure on $Y^G$ for which $\t{\mu}(\t{X}) = 1$.

Let $X\to \Pr\,Y:x\mapsto\nu_x$ be a disintegration of $\nu$ over the map $\b:Y\to X$.  Using this, define new probability measures for each $x \in X$ as follows. First, for each $g \in G$, define $\t{\nu}_{g,x}$ on $Y^{gH}$ by
 \[\t{\nu}_{g,x} := \int_Y \delta_{(S^{h^{-1}}y)_{gh \in gH}}\,\nu_x(\d y).\]
Now let $C \subseteq G$ be a cross-section for the space $G/H$ of left-cosets, identify $Y^G = \prod_{c \in C}Y^{cH}$, and on this product define
\[\t{\nu}_x := \bigotimes_{c \in C}\t{\nu}_{c,T^{c^{-1}}x}.\]
One may easily write down the finite-dimensional marginals of $\t{\nu}_x$ directly.  If $c_1,\ldots,c_m \in C$, and $h_{i,1}$, \ldots, $h_{i,n_i} \in H$ for each $i \leq m$, and also $A_{i,j} \in \S_Y$ for all $i \leq m$ and $j\leq n_i$, then
\begin{eqnarray}\label{eq:margs}
&&\t{\nu}_x\big\{(y_g)_g\,\big|\ y_{c_ih_{i,j}} \in A_{i,j}\ \forall i\leq m,\,j\leq n_i\big\} \nonumber\\
&&= \prod_{i=1}^m\t{\nu}_{c_i,T^{c_i^{-1}}x}\big\{(y_{c_ih})_{h \in H}\,\big|\ y_{c_ih_{i,j}} \in A_{i,j}\ \forall j \leq n_i\big\} \nonumber\\
&&= \prod_{i=1}^m\nu_{T^{c_i^{-1}}x}\big(S^{h_{i,1}}(A_{i,1})\cap \cdots \cap S^{h_{i,n_i}}(A_{i,n_i})\big).
\end{eqnarray}

The following basic properties of $\t{\nu}_x$ are now easily checked:
\begin{itemize}
\item[i)] If $g_1H = g_2H$, say with $g_1 = g_2h_1$, and $x \in X$, then
\begin{eqnarray*}
\t{\nu}_{g_1,T^{g_1^{-1}}x} &=& \int_Y \delta_{(S^{h^{-1}}y)_{g_1h \in g_1H}}\ \nu_{T^{g_1^{-1}}x}(\d y)\\ &=& \int_Y \delta_{(S^{h^{-1}}y)_{g_2h_1h \in g_2H}}\ \nu_{T^{h_1^{-1}}T^{g_2^{-1}}x}(\d y)\\
&=& \int_Y \delta_{(S^{h^{-1}}y)_{g_2h_1h \in g_2H}}\ (S^{h_1^{-1}}_\ast\nu_{T^{g_2^{-1}}x})(\d y)\\
&=& \int_Y \delta_{(S^{h^{-1}}S^{h_1^{-1}}y)_{g_2h_1h \in g_2H}}\ \nu_{T^{g_2^{-1}}x}(\d y)\\
&=& \t{\nu}_{g_2,T^{g_2^{-1}}x}.
\end{eqnarray*}
It follows that $\t{\nu}_x$ does not depend on the choice of cross-section $C$, and the formula~(\ref{eq:margs}) holds with any choice of $C$.
\item[ii)] For each $g \in G$, say $g = ch \in cH$, the marginal of $\t{\nu}_x$ on coordinate $g$ is
\[S^{h^{-1}}_\ast\nu_{T^{c^{-1}}x} = \nu_{T^{h^{-1}}T^{c^{-1}}x} = \nu_{T^{g^{-1}}x}.\]
\item[iii)] If $(y_g)_g$ is sampled at random from $\t{\nu}_x$ and $g \in cH$, then $y_c$ a.s. determines the whole tuple $(y_{ch})_{ch \in cH}$: specifically,
\[y_{ch} = S^{h^{-1}}y_c \quad \hbox{a.s.}\]
\end{itemize}
It also holds that if $g_1$, \ldots, $g_m$ lie in distinct left-cosets of $H$ and $(y_g)_g \sim \t{\nu}_x$, then the coordinates $y_{g_1}$, \ldots, $y_{g_m}$ are independent, but we will not need this fact.

Finally, let
\[\t{\mu} := \int_X \t{\nu}_x\,\mu(\d x).\]

Recalling the definition of $\t{X}$, properties (ii) and (iii) above imply that $\t{\nu}_x(\t{X}) = 1$ for all $x$, and hence also $\t{\mu}(\t{X}) = 1$.

We have seen that the left-regular representation defines an action of $G$ on $\t{X}$, and the required triangular diagram commutes by the definition of $\pi$, so it remains to check the following.
\begin{itemize}
\item (The new $G$-space $(\t{X},\t{\mu},\t{T})$ is probability-preserving.)  Suppose that $k \in G$ and $x \in X$, that $c_1,\ldots,c_m \in C$, that $h_{i,1}$, \ldots, $h_{i,n_i} \in H$ for each $i \leq m$, and that $A_{i,j} \in \S_Y$ for all $i \leq m$ and $j\leq n_i$. Then one has
\begin{eqnarray*}
&&\t{T}^k_\ast\t{\nu}_x\big\{(y_g)_g\,\big|\ y_{c_ih_{i,j}} \in A_{i,j}\ \forall i\leq m,\,j\leq n_i\big\} \\ 
&&=
\t{\nu}_x\big\{\t{T}^{k^{-1}}(y_g)_g\,\big|\ y_{c_ih_{i,j}} \in A_{i,j}\ \forall i\leq m,\,j\leq n_i\big\} \\
&&=
\t{\nu}_x\big\{(y_g)_g\,\big|\ y_{k^{-1}c_ih_{i,j}} \in A_{i,j}\ \forall i\leq m,\,j\leq n_i\big\}
\end{eqnarray*}
Since $C$ is a cross-section of $G/H$, so is $k^{-1}C$.  We may therefore apply~(\ref{eq:margs}) with the cross-section $k^{-1}C$ to deduce that the above is equal to
\[\prod_{i=1}^m\nu_{T^{c_i^{-1}k}x}\big(S^{h_{i,1}}(A_{i,1})\cap \cdots \cap S^{h_{i,n_i}}(A_{i,n_i})\big).\]
On the other hand, equation~(\ref{eq:margs}) applied with the cross-section $C$ gives that this is equal to
\[\t{\nu}_{T^kx}\big\{(y_g)_g\,\big|\ y_{c_ih_{i,j}} \in A_{i,j}\ \forall i\leq m,\,j\leq n_i\big\}.\]
Therefore $\t{T}^k_\ast\t{\nu}_x = \t{\nu}_{T^kx}$, and integrating this over $x$ gives $\t{T}^k_\ast\t{\mu} = \t{\mu}$.

\item (The map $\a$ defines a factor map of $H$-spaces.)  If $h \in H$ and $(y_g)_g \in \t{X}$, then
\[\a(\t{T}^h((y_g)_g)) = \a((y_{h^{-1}g})_g) = y_{h^{-1}} = S^hy_e = S^h\a((y_g)_g),\]
where the penultimate equality is given by property (iii) above.  Also, property (ii) above gives that
\[\a_\ast\t{\mu} = \int_X\a_\ast\t{\nu}_x\,\mu(\d x) = \int_X \nu_x\,\mu(\d x) = \nu.\]

\item (The map $\pi$ defines a factor map of $G$-spaces.)  If $k \in G$ and $(y_g)_g \in \t{X}$, then
\[\pi(\t{T}^k((y_g)_g)) = \b(\a((y_{k^{-1}g})_g)) = \b(y_{k^{-1}}).\]
If $x\in X$ and $(y_g)_g \sim \t{\nu}_x$, then property (ii) above gives that $y_{k^{-1}} \sim \nu_{T^kx}$, and hence $\b(y_{k^{-1}}) = T^kx = T^k\b(y_e)$ a.s.  Since this holds for every $x$, integrating over $x$ gives
\[\pi(\t{T}^k((y_g)_g)) = T^k\pi((y_g)_g) \quad \hbox{a.s.}\]
Also, another appeal to property (ii) above gives
\[\pi_\ast\t{\mu} = \int_X \pi_\ast\t{\nu}_x\,\mu(\d x) = \int_X \b_\ast\nu_x\,\mu(\d x) = \int_X \delta_x\,\mu(\d x) = \mu.\]
\end{itemize}
\end{proof}

\section{Functorial $\s$-subalgebras and subspaces, and satedness}\label{sec:func}

\begin{dfn}[Functorial $\s$-subalgebras and subspaces]
Given $G$, a \textbf{functorial $\s$-subalgebra of $G$-spaces} is a map $\sfF$ which to any $G$-space $\bfX = (X,\mu,T)$ assigns a $\mu$-complete $\s$-subalgebra
\[\S^\sfF_\bfX \subseteq \S_\bfX,\]
and such that for any $G$-extension $\pi:\bfX\to \bfY$ one has
\[\S^\sfF_\bfX \supseteq \pi^{-1}(\S^\sfF_\bfY).\]

Similarly, a \textbf{functorial $L^2$-subspace of $G$-spaces} is a map $\sfV$ which to each $G$-space $\bfX = (X,\mu,T)$ assigns a closed subspace
\[\sfV_\bfX \leq L^2(\mu),\]
and such that for any $G$-extension $\pi:\bfX\to \bfY$ one has
\[\sfV_\bfX \geq \sfV_\bfY\circ \pi := \{f\circ \pi\,|\ f \in \sfV_\bfY\}.\]
In this setting, $P^\sfV_\bfX: L^2(\mu) \to \sfV_\bfX$ will denote the orthogonal projection onto $\sfV_\bfX$.
\end{dfn}

The above behaviour relative to factors is called the \textbf{functoriality} of $\sfF$ or $\sfV$.  Its first consequence is that $\sfF$ and $\sfV$ respect isomorphisms of $G$-spaces: if $\a:\bfX\stackrel{\cong}{\to} \bfY$, then
\[\S^\sfF_\bfX = \a^{-1}(\S^\sfF_\bfY)\]
(where strict equality holds owing to the assumption that these $\s$-algebras are both $\mu$-complete) and
\[\sfV_\bfX = \sfV_\bfY\circ \a.\]

\begin{ex}
If $H\leq G$ is any subgroup, then the map $\bfX \mapsto \S^H_\bfX$ (the $\s$-subalgebra of $H$-almost-invariant sets) defines a functorial $\s$-subalgebra of $G$-spaces.  In case $H \unlhd G$, this actually defines a factor of $\bfX$, but otherwise it may not: in general, one has
\[T^g(\S^H_\bfX) = \S^{gHg^{-1}}_\bfX.\]

This class of examples will provide the building blocks for all of the other functorial $\s$-subglebras that we meet later. \fin
\end{ex}

If $\sfF$ is a functorial $\s$-subalgebra of $G$-spaces, then setting
\[\sfV_\bfX := L^2(\mu|\S^\sfF_\bfX)\]
defines a functorial $L^2$-subspace of $G$-spaces, where this denotes the subspace of $L^2(\mu)$ generated by the $\S^\sfF_\bfX$-measurable functions.  In this case, $P^\sfV_\bfX$ is the operator of conditional expectation onto $\S^\sfF_\bfX$. However, not all functorial $L^2$-subspaces arise in this way.  For instance, given any two functorial $L^2$-subspaces $\sfV_1$, $\sfV_2$ of $G$-spaces, a new functorial $L^2$-subspace may be defined by
\[\sfV_\bfX := \ol{\sfV_{1,\bfX} + \sfV_{2,\bfX}}.\]
If $H_1$, $H_2 \leq G$, then this gives rise to the example
\[\sfV_\bfX := \ol{L^2(\mu|\S^{H_1}_\bfX) + L^2(\mu|\S^{H_2}_\bfX)}.\]
The elements of this subspace generate the functorial $\s$-algebra $\S^{H_1}_\bfX\vee \S^{H_2}_{\bfX}$, but in general one may have
\[\ol{L^2(\mu|\S^{H_1}_\bfX) + L^2(\mu|\S^{H_2}_\bfX)} \lneqq L^2(\mu|\S^{H_1}_\bfX\vee \S^{H_2}_\bfX).\]

In fact, the functorial $L^2$-subspaces that appear later in this work will all correspond to functorial $\s$-subalgebras.  However, the theory of satedness depends only on the subspace structure, so it seems appropriate to develop it in that generality.

To prepare for the next definition, recall that if $\frK_1,\frK_2 \leq \frH$ are two closed subspaces of a real Hilbert space, and $\frL \leq \frK_1 \cap \frK_2$ is a common further closed subspace, then $\frK_1$ and $\frK_2$ are \textbf{relatively orthogonal} over $\frL$ if
\[\langle u,v\rangle = \langle P_{\frL}u,P_{\frL}v\rangle \quad \forall u \in \frK_1,\ v \in\frK_2,\]
where $P_{\frL}$ is the orthogonal projection onto $\frL$.  This requires that in fact $\frL = \frK_1\cap \frK_2$, and is equivalent to asserting that $P_{\frK_2}u = P_\frL u$ for all $u \in \frK_1$, and vice-versa.  Clearly it suffices to verify this for elements drawn from any dense subsets of $\frK_1$ and $\frK_2$.

\begin{dfn}[Satedness]
Let $\sfV$ be a functorial $L^2$-subspace of $G$-spaces.  A $G$-space $\bfX = (X,\mu,T)$ is \textbf{$\sfV$-sated} if the following holds: for any $G$-extension $\bfY = (Y,\nu,S)\stackrel{\xi}{\to} (X,\mu,T)$, the subspaces $L^2(\mu)\circ \xi$ and $\sfV_\bfY$ are relatively orthogonal over their common further subspace $\sfV_\bfX \circ \xi$.

More generally, a $G$-extension $\t{\bfX}\stackrel{\pi}{\to}\bfX$ is \textbf{relatively $\sfV$-sated} if the following holds: for any further $G$-extension $\bfY\stackrel{\xi}{\to}\t{\bfX}$, the subspaces $L^2(\mu)\circ (\pi \circ \xi)$ and $\sfV_\bfY$ are relatively orthogonal over $\sfV_{\t{\bfX}}\circ \pi$.
\end{dfn}

Clearly a $G$-space $\bfX$ is $\sfV$-sated if and only if $\bfX\stackrel{\rm{id}}{\to} \bfX$ is relatively $\sfV$-sated.  In case $\sfV_\bfX = L^2(\mu|\S^\sfF_\bfX)$ for some functorial $\s$-algebra $\sfF$, one may write that a $G$-space or $G$-extension is \textbf{$\sfF$-sated}, rather than $\sfV$-sated.  For a $G$-space $\bfX = (X,\mu,T)$, this asserts that for any $G$-extension $\xi:\bfY = (Y,\nu,S)\to \bfX$, the $\s$-subalgebras
\[\xi^{-1}(\S_X) \quad \hbox{and} \quad \S^\sfF_\bfY\]
are relatively independent over $\xi^{-1}(\S^\sfF_\bfX)$.

The key feature of satedness is that all $G$-spaces have sated extensions.  This generalizes the corresponding result for satedness relative to idempotent classes~(\cite[Theorem 2.3.2]{Aus--thesis}).  The proof here will be a near-verbatim copy of that one, once we have the following auxiliary lemmas.

\begin{lem}\label{lem:bigger-extn-still-rel-sated}
If $\t{\bfX}\stackrel{\pi}{\to} \bfX$ is a relatively $\sfV$-sated $G$-extension, and $\bfZ = (Z,\theta,R)\stackrel{\a}{\to} \t{\bfX}$ is a further $G$-extension, then $\bfZ \stackrel{\a\circ \pi}{\to} \bfX$ is also relatively $\sfV$-sated.
\end{lem}

\begin{proof}
Suppose that $\bfY = (Y,\nu,S)\stackrel{\xi}{\to} \bfZ$ is another $G$-extension, and that $f \in L^2(\mu)$ and $g \in \sfV_\bfY$.  Then applying the definition of relative satedness to the composed extension $\bfY\stackrel{\a\circ \xi}{\to} \t{\bfX}$ gives
\[\int_Y(f\circ \pi\circ \a\circ \xi)\cdot g\,\d\nu = \int_Y (P^\sfV_{\t{\bfX}}(f\circ\pi)\circ \a \circ \xi)\cdot g\,\d\nu.\]

This will turn into the required equality of inner products if we show that
\[(P^\sfV_{\t{\bfX}}(f\circ \pi))\circ \a = P^\sfV_\bfZ(f\circ \pi\circ \a).\]

However, in light of the inclusion $\sfV_{\t{\bfX}}\circ \a \subseteq \sfV_\bfZ$ and standard properties of orthogonal projection, this is equivalent to the equality
\[\int_Z(P^\sfV_{\t{\bfX}}(f\circ \pi)\circ \a)\cdot h\,\d\theta = \int_Z(f\circ \pi\circ \a)\cdot h\,\d\theta \quad \forall h \in \sfV_\bfZ,\]
and this is precisely the relative $\sfV$-satedness of $\pi$ applied to $\a$.
\end{proof}

\begin{lem}\label{lem:inv-lim-sated}
If
\[\cdots \stackrel{\pi_2}{\to} \bfX_2 \stackrel{\pi_1}{\to}\bfX_1 \stackrel{\pi_0}{\to} \bfX_0\]
is an inverse sequence of $G$-spaces in which each $\pi_i$ is relatively $\sfV$-sated, and if $\bfX_\infty$, $(\psi_m)_m$ is the inverse limit of this sequence, then $\bfX_\infty$ is $\sfV$-sated.
\end{lem}

\begin{proof}
First, all the resulting $G$-extensions $\bfX_\infty \stackrel{\psi_m}{\to} \bfX_m$ are relatively $\sfV$-sated,
because we may factorize $\psi_m = \pi_m\circ\psi_{m+1}$ and then apply Lemma~\ref{lem:bigger-extn-still-rel-sated}.  However, this now implies that for any further $G$-extension $\bfY\stackrel{\xi}{\to} \bfX_\infty$ and for $\pi := \rm{id}_{\t{X}}$, we have
\[\int_Y (f\circ \xi)\cdot g\,\d\nu = \int_Y ((P^\sfV_{\bfX_\infty}f)\circ \xi)\cdot g\,\d\nu \]
for all $g \in \sfV_\bfY$ and all $f \in \bigcup_{m\geq 1}(L^2(\mu_m)\circ \psi_m)$. Since this last union is dense in $L^2(\mu_\infty)$, this completes the proof.
\end{proof}

\begin{thm}\label{thm:sateds-exist}
If $\sfV$ is a functorial $L^2$-subspace of $G$-spaces, then every $G$-space has a $\sfV$-sated extension.
\end{thm}

\begin{proof}
Let $\bfX = (X,\mu,T)$ be a $G$-space.

\vspace{7pt}

\emph{Step 1}\quad We first show that $\bfX$ has a relatively $\sfV$-sated extension.  This uses the same `energy increment'
argument as in~\cite{Aus--thesis}.

Let $\{f_r\,|\ r\geq 1\}$ be a countable dense subset of the unit ball of $L^2(\mu)$, and let $(r_i)_{i\geq 1}$ be a member
of $\bbN^\bbN$ in which every non-negative integer appears
infinitely often.

We will now construct an inverse sequence $(\bfX_m)_{m\geq 0}$,
$(\psi^m_k)_{m\geq k \geq 0}$ by the following recursion.  First let $\bfX_0 := \bfX$.  Then, supposing that for some $m_1 \geq 0$ we have already
obtained $(\bfX_m)_{m=0}^{m_1}$, $(\psi^m_k)_{m_1 \geq
m\geq k\geq 0}$, let $\psi^{m_1+1}_{m_1}:\bfX_{m_1+1}\to \bfX_{m_1}$ be an extension such that the difference
\[\|P^\sfV_{\bfX_{m_1+1}}(f_{r_{m_1}}\circ \psi^{m_1+1}_0)\|_2 - \|P^\sfV_{\bfX_{m_1}}(f_{r_{m_1}}\circ \psi^{m_1}_0\,)\|_2\]
is at least half its supremal possible value over all extensions of $\bfX_{m_1}$, where of course we let $\psi^{m_1+1}_0 := \psi^{m_1}_0\circ \psi^{m_1+1}_{m_1}$.

Let $\bfX_\infty$, $(\psi_m)_{m \geq 0}$ be the
inverse limit of this sequence.  We will show that $\bfX_\infty\stackrel{\psi_0}{\to}\bfX$ is relatively $\sfV$-sated. Letting $\pi:\bfY\to\bfX_\infty$
be an arbitrary further extension, this is equivalent to showing that
\[P^\sfV_{\bfY}(f\circ\psi_0\circ \pi) =
P^\sfV_{\bfX_\infty}(f\circ \psi_0)\circ\pi \quad \forall f\in L^2(\mu).\]

It suffices to prove this for every $f_r$ in our previously-chosen dense subset.  Also, since $\sfV_{\bfY} \supseteq
\sfV_{\bfX_\infty}\circ \pi$, the result will follow if we only
show that
\[\|P^\sfV_{\bfY}(f_r\circ\psi_0\circ\pi)\|_2 =
\|P^\sfV_{\bfX_\infty}(f_r\circ\psi_0)\|_2.\]
Suppose, for the sake of contradiction, that the left-hand norm here were strictly larger.  The sequence of norms
\[\|P^\sfV_{\bfX_m}(f_r\circ\psi^m_0)\|_2\]
is non-decreasing as $m\to\infty$, and bounded above by $\|f_r\|_2$.  Therefore it would follow that for some sufficiently large $m$ we would have $r_m = r$ (since
each integer appears infinitely often as some $r_m$) but also
\begin{multline*}
\|P^\sfV_{\bfX_{m+1}}(f_r\circ\psi^{m+1}_0)\|_2
- \|P^\sfV_{\bfX_m}(f_r\circ\psi^m_0)\|_2\\
< \frac{1}{2}\Big(\|P^\sfV_{\bfY}(f_r\circ\psi_0\circ\pi)\|_2 - \|P^\sfV_{\bfX_\infty}(f\circ\psi_0)\|_2\Big)\\
\leq \frac{1}{2}\Big(\|P^\sfV_{\bfY}(f_r\circ\psi_0\circ\pi)\|_2 - \|P^\sfV_{\bfX_m}(f\circ\psi^m_0)\|_2\Big).
\end{multline*}
This would contradict the choice of $\bfX_{m+1}\to
\bfX_m$ in our construction above, so we must actually have the equality of
$L^2$-norms required.

\vspace{7pt}

\emph{Step 2}\quad Iterating the construction of Step 1, we may let
\[\cdots \stackrel{\pi_2}{\to} \bfX_2 \stackrel{\pi_1}{\to} \bfX_1 \stackrel{\pi_0}{\to} \bfX\]
be an inverse seqeuence in which each extension $\pi_i$ is relatively $\sfV$-sated.  Letting $\bfX_\infty$, $(\pi_m)_{m\geq 0}$ be its inverse limit, Lemma~\ref{lem:inv-lim-sated} completes the proof.
\end{proof}

\begin{cor}\label{cor:mult-sateds-exist}
Let $\sfV_1$, $\sfV_2$, \ldots be any countable family of functorial $L^2$-subspaces of $G$-spaces.  Then every $G$-space has an extension which is simultaneously $\sfV_r$-sated for every $r$.
\end{cor}

\begin{proof}
Let $(r_i)_i$ be an element of $\bbN^\bbN$ in which every positive integer appears infinitely often.  By repeatedly implementing Theorem~\ref{thm:sateds-exist}, let
\[\cdots \stackrel{\pi_2}{\to} \bfX_2 \stackrel{\pi_1}{\to} \bfX_1 \stackrel{\pi_0}{\to} \bfX\]
be an inverse sequence in which each $\bfX_i$ is $\sfV_{r_i}$-sated.  Also, let $\pi^n_m := \pi_m\circ \cdots \circ \pi_{n-1}$ whenever $m < n$.  Finally, let $\bfX_\infty$ be the inverse limit of this sequence.  Then for each $r \geq 1$, there is an infinite subsequence $i_1(r) < i_2(r) < \ldots $ in $\bbN$ such that $r_{i_1(r)} = r_{i_2(r)} = \cdots = r$, and $\bfX_\infty$ may be identified with the inverse limit of the thinned-out inverse sequence
\[\cdots \stackrel{\pi^{i_3(r)}_{i_2(r)}}{\to} \bfX_{i_2(r)} \stackrel{\pi^{i_2(r)}_{i_1(r)}}{\to} \bfX_{i_1(r)} \stackrel{\pi^{i_1(r)}_0}{\to} \bfX.\]
Given this, Lemma~\ref{lem:inv-lim-sated} implies that $\bfX_\infty$ is $\sfV_r$-sated.  Since $r$ was arbitrary, this completes the proof.
\end{proof}

\section{Characteristic subspaces and proof of convergence}

\subsection{Subgroups associated to commuting tuples of actions}

We now begin to work with commuting tuples of $G$-actions.  We will need to call on several different subgroups of $G^d$ in the sequel, so the next step is to set up some bespoke notation for handling them.  We will sometimes use a boldface $\bf{g}$ to denote a tuple $(g_i)_{i=1}^d$ in $G^d$, and will denote the identity element of $G$ by $1_G$.

\begin{dfn}
Fix $G$ and $d$, and let $e = \{i_1 < \ldots < i_r\} \subseteq [d]$ with $r\geq 2$ and also $\{i < j\} \subseteq [d]$.  Then we define
\[H_e := \{\bf{g}\in G^d\,|\ g_{i_s+1} = g_{i_s+2} = \ldots = g_{i_{s+1}}\ \hbox{for each}\ s=1,\ldots,r-1\},\]
\[K_{\{i,j\}} := \{\bf{g} \in H_{\{i,j\}}\,|\ g_\ell = 1_G \ \forall \ell \in (i;j]\},\]
and
\[L_e := \{\bf{g} \in H_e\,|\ g_i = 1_G \ \forall i \in [d]\setminus (i_1;i_r]\}.\]
\end{dfn}

Routine calculations give the following basic properties.

\begin{lem}\label{lem:subgp-props}
\begin{enumerate}
\item[(1)] For $e = \{i_1 < \ldots < i_r\}$ as above, the subgroups $L_e$ and $K_{\{i_1,i_r\}}$ commute and generate $H_e$.
\item[(2)] If $a \subseteq e \subseteq [k]$ with $|a| \geq 2$, then $L_a \leq L_e$.

\item[(3)] If $a \subseteq e \subseteq [k]$ with $|a| \geq 2$, and also
\[e \cap [\min a;\max a] = a,\]
then $L_a \unlhd H_e$.  In particular, $L_e \unlhd H_e$. \qed
\end{enumerate}
\end{lem}

Part (3) of this lemma has the following immediate consequence.

\begin{cor}\label{cor:some-invce}
If $a \subseteq e \subseteq [k]$ with $|a| \geq 2$, and also
\[e \cap [\min a;\max a] = a,\]
then $\S^{L_a}_\bfX$ is globally $H_e$-invariant. \qed
\end{cor}

\subsection{The Host-Kra inequality}

In order to show that a suitably-sated $G$-space has some other desirable property, one must find an extension of it for which the relative independence given by satedness implies that other property.  The key to such a proof is usually constructing the right extension.

Where satedness was used in the previous works~\cite{Aus--nonconv} and~\cite{Aus--newmultiSzem}, that extension could be constructed directly from the Furstenberg self-joining arising from some non-conventional averages.  However, this seems to be more problematic in the present setting, and we will take a different approach.  The construction below is a close analog of the construction by Host and Kra of certain `cubical' extensions of a $\bbZ$-space in~\cite{HosKra05}.  That machinery has also been extended by Host to commuting tuples of $\bbZ$-actions in~\cite{Hos09}.

Fix now a $G^d$-space $\bfX = (X,\mu,T_1,\ldots,T_d)$, and let $\bfY^{(0)} := \bfX$.  Our next step is to construct recursively a height-$(d+1)$ tower of new probability-preserving $G^d$-spaces, which we shall denote by
\begin{eqnarray}\label{eq:tower}
\bfY^{(d)} \stackrel{\xi^{(d)}}{\to} \bfY^{(d-1)} \stackrel{\xi^{(d-1)}}{\to} \cdots \stackrel{\xi^{(2)}}{\to} \bfY^{(1)} \stackrel{\xi^{(1)}}{\to} \bfY^{(0)} = \bfX.
\end{eqnarray}

The construction will also give some other auxiliary $G^d$-spaces $\bfZ^{(j)}$, and they too will be used later.

Supposing the tower has already been constructed up to some level $j \leq d-1$, the next extension is constructed in the following steps.
\begin{itemize}
\item[i)] From $\bfY^{(j)} = (Y^{(j)},\nu^{(j)},S^{(j)})$, define a new $H_{\{d-j-1,d\}}$-action $\t{S}^{(j)}$ on the same space by setting
\begin{eqnarray}\label{eq:comm1}
(\t{S}^{(j)}_i)^g := (S^{(j)}_i)^g \quad \forall g \in G,\ i < d-j-1,
\end{eqnarray}
\begin{eqnarray}\label{eq:comm2}
(\t{S}^{(j)}_{d-j-1})^g := (S^{(j)}_{[d-j-1;d]})^g \quad \forall g \in G,
\end{eqnarray}
and
\begin{eqnarray}\label{eq:comm3}
(\t{S}^{(j)}_{(d-j-1;d]})^g := \rm{id} \quad \forall g \in G
\end{eqnarray}
(with the understanding that~(\ref{eq:comm1}) and~(\ref{eq:comm2}) are vacuous in case $j=d-1$).

\item[ii)] Now consider the $H_{\{d-j-1,d\}}$-space
\begin{eqnarray*}
\bfZ^{(j+1)} &=& (Z^{(j+1)},\theta^{(j+1)},R^{(j+1)})\\
&:=& \big(Y^{(j)}\times Y^{(j)},\ \nu^{(j)}\otimes_{\S^{L_{\{d-j-1,d\}}}_{\bfY^{(j)}}} \nu^{(j)},\ (S^{(j)})^{\uhr H_{\{d-j-1,d\}}}\times \t{S}^{(j)}\big).
\end{eqnarray*}
Let $\xi^{(j+1)}_0,\xi^{(j+1)}_1:Z^{(j+1)}\to Y^{(j)}$ be the two coordinate projections.  They are both factor maps of $H_{\{d-j-1,d\}}$-spaces. Notice that $\theta^{(j+1)}$ is $R^{(j+1)}$-invariant because both of the actions $(S^{(j)})^{\uhr H_{\{d-j-1,d\}}}$ and $\t{S}^{(j)}$ preserve the $\s$-subalgebra $\S^{L_{\{d-j-1,d\}}}_{\bfY^{(j)}}$, by Corollary~\ref{cor:some-invce}.

\item[iii)] Finally, let $\bfY^{(j+1)}\stackrel{\xi^{(j+1)}}{\to} \bfY^{(j)}$ be an extension of $G^d$-spaces for which there is a commutative diagram
\begin{center}
$\phantom{i}$\xymatrix{
(\bfY^{(j+1)})^{\uhr H_{\{d-j-1,d\}}} \ar_{\a^{(j+1)}}[dr]\ar^{\xi^{(j+1)}}[rr] && (\bfY^{(j)})^{\uhr H_{\{d-j-1,d\}}}\\
& \bfZ^{(j+1)}, \ar_{\xi^{(j+1)}_0}[ur] }
\end{center}
as provided by Theorem~\ref{thm:recoverG}.
\end{itemize}

Having made this construction, for each $j\in \{1,2,\ldots,d\}$ we also define a family of maps
\[\pi^{(j)}_\eta:Y^{(j)} \to X\]
indexed by $\eta \in \{0,1\}^j$, by setting
\[\pi^{(j)}_{(\eta_1,\ldots,\eta_j)}:= \xi^{(1)}_{\eta_1}\circ\a^{(1)}\circ \xi^{(2)}_{\eta_2}\circ \a^{(2)} \circ \cdots\circ \xi^{(j)}_{\eta_j}\circ \a^{(j)}.\]
Clearly $(\pi^{(j)}_\eta)_\ast \nu^{(j)} = \mu$ for every $\eta$.   Also,
\[\pi^{(j)}_{0^j} = \xi^{(1)}\circ\cdots\circ \xi^{(j)}:\bfY^{(j)}\to \bfX\]
is a factor map of $G^d$-spaces, where $0^j:= (0,0,\ldots,0) \in \{0,1\}^j$.

\begin{lem}\label{lem:intertwine}
Let $r \in [d]$, let $\eta \in \{0,1\}^r\setminus\{0^r\}$, and let $\ell \in [r]$ be maximal such that $\eta_\ell = 1$.  Then $\pi^{(r)}_\eta$ satisfies the following intertwining relations
\begin{eqnarray}\label{eq:comm4}
\pi^{(r)}_\eta\circ S^{(r)}_i = T_i\circ \pi^{(r)}_\eta \quad \forall i < d-\ell,
\end{eqnarray}
\begin{eqnarray}\label{eq:comm5}
\pi^{(r)}_\eta\circ S^{(r)}_{d-\ell} = T_{[d-\ell;d]}\circ \pi^{(r)}_\eta
\end{eqnarray}
and
\begin{eqnarray}\label{eq:comm6}
\pi^{(r)}_\eta\circ S^{(r)}_{(d-\ell;d]} = \pi^{(r)}_\eta.
\end{eqnarray}
\end{lem}

\begin{rmk}
There are no such simple relations for the compositions $\pi^{(r)}_\eta\circ S^{(r)}_i$ when $i \geq d-\ell+1$, but we will not need these. \fin
\end{rmk}

\begin{proof}
By the definition of $\ell$, for this $\eta$ we may write $\pi^{(r)}_\eta = \pi'\circ\pi''$, where
\begin{multline}\label{eq:dfn-pieta}
\pi' := \pi^{(\ell)}_{(\eta_1,\ldots,\eta_\ell)} = \xi^{(1)}_{\eta_1}\circ\a^{(1)}\circ \xi^{(2)}_{\eta_2}\circ \a^{(2)}\circ \cdots\circ \xi^{(\ell)}_1\circ \a^{(\ell)}\\ \hbox{and}\quad  \pi'' := \xi^{(\ell+1)}\circ \cdots\circ \xi^{(r)}.
\end{multline}

All three of the desired relations concern the actions of subgroups of $H_{\{d-\ell,d\}}$, and all the maps in the compositions in~(\ref{eq:dfn-pieta}) are factor maps of $H_{\{d-\ell,d\}}$-spaces.  We will read off the desired results from the simpler relations~(\ref{eq:comm1}),~(\ref{eq:comm2}) and~(\ref{eq:comm3}).

In the first place, each $\xi^{(j)}$ appearing in the definition of $\pi''$ actually intertwines the whole $G^d$-actions by construction, so
\[\pi''\circ S_i^{(r)} = S_i^{(\ell)}\circ \pi'' \quad \forall i \in [d].\]

It therefore suffices to prove that $\pi' \circ S^{(\ell)}_i = T_i\circ \pi'$ for all $i < d - \ell$, and similarly for the other two desired relations.

\vspace{7pt}

\emph{Step 1.}\quad Suppose that $i \leq d-\ell$ and $j \leq \ell$. Then the definitions of $\a^{(j)}$, $\xi_0^{(j)}$ and $\xi_1^{(j)}$ give
\[\xi^{(j)}_\eta\circ\a^{(j)}\circ S_i^{(j)} = \xi^{(j)}_\eta\circ R^{(j)}_i \circ \a^{(j)}= \left\{\begin{array}{ll} S^{(j-1)}_i\circ \xi^{(j)}_\eta\circ\a^{(j)} & \quad \hbox{if}\ \eta = 0\\
 \t{S}^{(j-1)}_i \circ \xi^{(j)}_\eta\circ\a^{(j)} & \quad \hbox{if}\ \eta = 1.\end{array}\right.\]

In case $i < d - \ell \leq d - j$, this is equal to $S^{(j)}_i\circ \xi^{(j)}_\eta\circ\a^{(j)}$ for either value of $\eta$, by~(\ref{eq:comm1}).  Applying this repeatedly for $j=\ell,\ell-1,\ldots,1$ in the composition that defines $\pi'$, we obtain
\[\pi'\circ S_i^{(\ell)} = T_i\circ \pi'.\]
As explained above, this proves~(\ref{eq:comm4}).

\vspace{7pt}

\emph{Step 2.}\quad The same calculation as above gives
\[\xi^{(\ell)}_1\circ \a^{(\ell)}\circ S^{(\ell)}_{d-\ell} = \xi^{(\ell)}_1\circ R^{(\ell)}_{d-\ell} \circ \a^{(\ell)}= \t{S}^{(\ell-1)}_{d-\ell}\circ \xi^{(\ell)}_1\circ \a^{(\ell)},\]
and now this is equal to $S^{(\ell-1)}_{[d-\ell,d]}\circ \xi^{(\ell)}_1\circ \a^{(\ell)}$, by~(\ref{eq:comm2}).

On the other hand, if $j \leq \ell-1$, then another call to the definitions of $\a^{(j)}$, $\xi_0^{(j)}$ and $\xi_1^{(j)}$ gives
\[\xi^{(j)}_\eta\circ\a^{(j)}\circ S_{[d-\ell,d]}^{(j)} = \xi^{(j)}_\eta\circ R^{(j)}_{[d-\ell,d]}\circ \a^{(j)} = \left\{\begin{array}{ll} S^{(j-1)}_{[d-\ell;d]}\circ \xi^{(j)}_\eta\circ\a^{(j)} & \quad \hbox{if}\ \eta = 0\\
 \t{S}^{(j-1)}_{[d-\ell;d]}\circ \xi^{(j)}_\eta\circ\a^{(j)} & \quad \hbox{if}\ \eta = 1.\end{array}\right.\]
This time, since $j \leq \ell-1$, one has
\begin{eqnarray*}
\t{S}^{(j-1)}_{[d-\ell;d]} &=& \t{S}^{(j-1)}_{d-\ell}\circ \t{S}^{(j-1)}_{d-\ell+1}\circ \cdots \circ \t{S}^{(j-1)}_{d-j}\circ \t{S}^{(j-1)}_{(d-j;d]}\\
&=& S^{(j-1)}_{d-\ell}\circ S^{(j-1)}_{d-\ell+1}\circ \cdots \circ S^{(j-1)}_{[d-j;d]}\circ \rm{id} = S^{(j-1)}_{[d-\ell;d]},
\end{eqnarray*}
using~(\ref{eq:comm1}) and~(\ref{eq:comm2}).  So we obtain
\[\xi^{(j)}_\eta\circ\a^{(j)}\circ S_{[d-\ell,d]}^{(j)} = S^{(j-1)}_{[d-\ell,d]} \circ \xi^{(j)}_\eta\circ\a^{(j)}\]
for all $j \leq \ell-1$ and either value of $\eta$.

Combining these two calculations gives
\[\pi'\circ S^{(\ell)}_{d-\ell} = (\xi^{(1)}_{\eta_1}\circ \a^{(1)}\circ \cdots \circ \xi^{(\ell-1)}_{\eta_{\ell-1}}\circ \a^{(\ell-1)})\circ S^{(\ell-1)}_{[d-\ell;d]} = T_{[d-\ell;d]}\circ \pi',\]
and hence~(\ref{eq:comm5}).

\vspace{7pt}

\emph{Step 3.}\quad Finally,~(\ref{eq:comm3}) gives
\[\xi^{(\ell)}_1\circ \a^{(1)}\circ S^{(\ell)}_{(d-\ell;d]} = \t{S}^{(\ell-1)}_{(d-\ell;d]}\circ \xi^{(\ell)}_1\circ \a^{(1)} = \xi^{(\ell)}_1\circ \a^{(1)},\]
from which~(\ref{eq:comm6}) follows immediately.
\end{proof}

\begin{cor}\label{cor:intertwine}
If $r \in [d]$, $\eta \in \{0,1\}^r$, and if $j \in [r]$ is such that $\eta_i = 0$ for all $i \geq j+1$, then $\pi^{(r)}_\eta$ satisfies the following intertwining relations
\begin{eqnarray}\label{eq:comm7}
\pi^{(r)}_\eta\circ S^{(r)}_{[d-j;d]} = T_{[d-j;d]}\circ \pi^{(r)}_\eta.
\end{eqnarray}\qed
\end{cor}

We next prove an estimate relating the multi-linear forms $\L_n$ in~(\ref{eq:Lambda}) to certain integrals over these new $G^d$-spaces $\bfY^{(j)}$.  This will be the key estimate that enables an appeal to satedness. The following theorem relies on an iterated application of the van der Corput estimate, and follows essentially the same lines as Theorem 12.1 in~\cite{HosKra05}.

\begin{thm}\label{thm:ineq}
Let $\bfX = (X,\mu,T_1,\ldots,T_d)$ be a $G^d$-space, let $1 \leq j \leq d$, and let the tower~(\ref{eq:tower}) and the maps $\pi^{(j)}_\eta:Y^{(j)}\to X$ for $\eta\in \{0,1\}^j$ be constructed as above.

For $f_{d-j+1}$, \ldots, $f_d \in L^\infty(\mu)$, let
\[\L_n^{(j)}(f_{d-j+1},\ldots,f_d) := \frac{1}{|F_n|}\sum_{g \in F_n}\prod_{i=d-j+1}^d(f_i\circ T_{[d-j+1;i]}^g).\]
If $f_{d-j+1}$, \ldots, $f_d$ are all uniformly bounded by $1$, then
\[\limsup_{n\to\infty}\|\L_n^{(j)}(f_{d-j+1},\ldots,f_d)\|_2 \leq  \Big(\int_{Y^{(j)}} \prod_{\eta \in \{0,1\}^j} (\cal{C}^{|\eta|}f_d\circ \pi^{(j)}_\eta)\,\d\nu^{(j)}\Big)^{2^{-j}},\]
where $|\eta| := \sum_i\eta_i \mod 1$ and $\cal{C}$ is the operator of complex conjugation.
\end{thm}

Note that $\L_n^{(d)} = \L_n$, the averages in~(\ref{eq:Lambda}).  The integral appearing on the right-hand side of the last inequality actually defines a seminorm of the function $f_d$: these are the adaptations of the Host-Kra seminorms to the present setting.  However, our approach will not emphasize the seminorm axioms.

\begin{proof}
This is proved by induction on $j$.

\vspace{7pt}

\emph{Step 1: Base case.}\quad When $j=1$ the Norm Ergodic Theorem for amenable groups gives
\[\L_n^{(1)}(f_d) \to \sfE_\mu(f_d\,|\,\S_\bfX^{T_d}) = \sfE_\mu(f_d\,|\,\S_\bfX^{L_{\{d-1,d\}}}) \quad \hbox{in}\ \|\cdot\|_2,\]
and the square of the norm of this limit is equal to
\begin{multline*}
\int_{X\times X} (f_d\otimes \ol{f_d})\,\d\big(\mu\otimes_{\S_\bfX^{L_{\{d-1,d\}}}}\mu\big) = \int_{Z^{(1)}}f_d\circ \xi_0^{(1)}\cdot\ol{f_d\circ \xi_1^{(1)}}\,\d\theta^{(1)}\\ = \int_{Y^{(1)}}f_d\circ \pi_0^{(1)}\cdot\ol{f_d\circ \pi_1^{(1)}}\,\d\nu^{(1)},
\end{multline*}
by the definition of $Y^{(1)}$ and $\nu^{(1)}$.

\vspace{7pt}

\emph{Step 2: Van der Corput estimate.}\quad Now suppose the result is known up to some $j-1 \in \{1,2,\ldots,d-1\}$.

By the amenable-groups version of the van der Corput estimate~(\cite[Lemma 4.2]{BerMcCZha97}), one has
\begin{multline}\label{eq:vdC}
\limsup_{n\to\infty}\|\L_n^{(j)}(f_{d-j+1},\ldots,f_d)\|^2_2\\
\leq \limsup_{m\to \infty}\frac{1}{|F_m|^2}\sum_{h,k \in F_m}\limsup_{n\to\infty}\Big|\frac{1}{|F_n|}\sum_{g \in F_n}\int_X \prod_{i=d-j+1}^d(f_i\circ T_{[d-j+1;i]}^{hg})\ol{(f_i\circ T_{[d-j+1;i]}^{kg})}\,\d\mu\Big|
\end{multline}
For fixed $h$ and $k$, we may use the $T^g_{d-j+1}$-invariance of $\mu$ to re-arrange the above integral as follows:
\begin{eqnarray*}
&&\Big|\frac{1}{|F_n|}\sum_{g \in F_n}\int_X \prod_{i=d-j+1}^d(f_i\circ T_{[d-j+1;i]}^{hg})\ol{(f_i\circ T_{[d-j+1;i]}^{kg})}\,\d\mu\Big|\\
&&= \Big|\frac{1}{|F_n|}\sum_{g \in F_n}\int_X (f_{d-j+1}\circ T_{d-j+1}^h)\cdot \ol{(f_{d-j+1}\circ T_{d-j+1}^k)}\\
&&\quad\quad\quad\quad\quad\quad\quad\quad \cdot\Big(\prod_{i=d-j+2}^d((f_i\circ T_{[d-j+1;i]}^h)\cdot \ol{(f_i\circ T_{[d-j+1;i]}^k)})\circ T_{(d-j+1;i]}^g\Big)\,\d\mu\Big|.
\end{eqnarray*}
(At this point we have made crucial use of the commutativity of the different actions $T_i$.)
By the Cauchy-Bunyakowski-Schwartz Inequality, this, in turn, is bounded above by
\begin{eqnarray*}
&&\|(f_{d-j+1}\circ T_{d-j+1}^h)\cdot \ol{(f_{d-j+1}\circ T_{d-j+1}^k)}\|_2\\
&&\quad\quad\quad\quad \cdot\Big\|\frac{1}{|F_n|}\sum_{g \in F_n}\prod_{i=d-j+2}^d((f_i\circ T_{[d-j+1;i]}^h)\cdot \ol{(f_i\circ T_{[d-j+1;i]}^k)})\circ T_{(d-j+1;i]}^g\Big\|_2\\
&&\leq \big\|\L_n^{(j-1)}\big((f_{d-j+2}\circ T_{[d-j+1;d-j+2]}^h)\cdot \ol{(f_{d-j+2}\circ T_{[d-j+1;d-j+2]}^k)},\\
&&\quad\quad\quad\quad\quad\quad\quad\quad\quad\quad\quad\quad\quad\quad\quad\quad \ldots,(f_d\circ T_{[d-j+1;d]}^h)\cdot \ol{(f_d\circ T_{[d-j+1;d]}^k)}\big)\big\|_2
\end{eqnarray*}
(since $\|f_{d-j+1}\|_\infty \leq 1$).

\vspace{7pt}

\emph{Step 3: Use of inductive hypothesis.}\quad Combining these inequalities and using the inductive hypothesis, this gives
\begin{multline}\label{eq:use-of-ind-hyp}
\limsup_{n\to\infty}\Big|\frac{1}{|F_n|}\sum_{g \in F_n}\int_X \prod_{i=d-j+1}^d(f_i\circ T_{[d-j+1;i]}^{hg})\ol{(f_i\circ T_{[d-j+1;i]}^{kg})}\,\d\mu\Big|\\
\leq \Big(\int_{Y^{(j-1)}}\prod_{\eta \in \{0,1\}^{j-1}}\big(\cal{C}^{|\eta|}((f_d\circ T_{[d-j+1;d]}^h)\cdot \ol{(f_d\circ T_{[d-j+1;d]}^k)})\circ \pi^{(j-1)}_\eta\big)\,\d\nu^{(j-1)}\Big)^{2^{-(j-1)}}
\end{multline}
for each $h$ and $k$.

To lighten notation, now let
\[F := \prod_{\eta \in \{0,1\}^{j-1}}(\C^{|\eta|}f_d\circ \pi_\eta^{(j-1)}).\]
In terms of this function, Corollary~\ref{cor:intertwine} with $r := j-1$ allows us to write
\begin{multline*}
\prod_{\eta \in \{0,1\}^{j-1}}(\C^{|\eta|}(f_d\circ T^h_{[d-j+1;d]})\circ \pi_\eta^{(j-1)})\\ = \prod_{\eta \in \{0,1\}^{j-1}}(\C^{|\eta|}f_d\circ \pi_\eta^{(j-1)}\circ (S^{(j-1)}_{[d-j+1;d]})^h) = F\circ (S^{(j+1)}_{[d-j+1;d]})^h,
\end{multline*}
and similarly
\[\prod_{\eta \in \{0,1\}^{j-1}}(\C^{|\eta|}\ol{(f_d\circ T^k_{[d-j+1;d]})}\circ \pi_\eta^{(j-1)}) = \ol{F}\circ (S^{(j-1)}_{[d-j+1;d]})^k.\]

\vspace{7pt}

\emph{Step 4: Finish.}\quad Substituting into the right-hand side of~(\ref{eq:use-of-ind-hyp}), one obtains
\begin{multline*}
\limsup_{n\to\infty}\Big|\frac{1}{|F_n|}\sum_{g \in F_n}\int_X \prod_{i=d-j+1}^d(f_i\circ T_{[d-j+1;i]}^{hg})\ol{(f_i\circ T_{[d-j+1;i]}^{kg})}\,\d\mu\Big|\\
\leq \Big(\int_{Y^{(j-1)}}(F\circ (S^{(j+1)}_{[d-j+1;d]})^h)\cdot (\ol{F}\circ (S^{(j-1)}_{[d-j+1;d]})^k)\,\d\nu^{(j-1)}\Big)^{2^{-(j-1)}}.
\end{multline*}
Inserting this back into~(\ref{eq:vdC}) and using H\"older's Inequality for the average over $(h,k)$, one obtains
\begin{eqnarray*}
&&\limsup_{n\to\infty}\|\L_n^{(j)}(f_{d-j+1},\ldots,f_d)\|^2_2\\
&&\leq \limsup_{m\to \infty}\frac{1}{|F_m|^2}\sum_{h,k \in F_m}\Big(\int_{Y^{(j-1)}}(F\circ (S^{(j+1)}_{[d-j+1;d]})^h)\cdot (\ol{F}\circ (S^{(j-1)}_{[d-j+1;d]})^k)\,\d\nu^{(j-1)}\Big)^{2^{-(j-1)}}\\
&&\leq \limsup_{m\to \infty}\Big(\frac{1}{|F_m|^2}\sum_{h,k \in F_m}\int_{Y^{(j-1)}}(F\circ (S^{(j+1)}_{[d-j+1;d]})^h)\cdot (\ol{F}\circ (S^{(j-1)}_{[d-j+1;d]})^k)\,\d\nu^{(j-1)}\Big)^{2^{-(j-1)}}.
\end{eqnarray*}

Finally, the averages on the last line here actually converge as $m\to\infty$, by the Norm Ergodic Theorem for amenable groups, giving
\begin{eqnarray*}
&&\limsup_{n\to\infty}\|\L_n^{(j)}(f_{d-j+1},\ldots,f_d)\|^2_2\\
&&\leq \Big(\int_{Y^{(j-1)}} \sfE_{\nu^{(j-1)}}\big(F\,\big|\,\S_{\bfY^{(j-1)}}^{L_{\{d-j,d\}}}\big)\cdot \sfE_{\nu^{(j-1)}}\big(\ol{F}\,\big|\,\S_{\bfY^{(j-1)}}^{L_{\{d-j,d\}}}\big)\,\d\nu^{(j-1)}\Big)^{2^{-(j-1)}}\\
&&= \Big(\int_{Z^{(j)}} F\circ \xi^{(j)}_0\cdot \ol{F\circ \xi^{(j)}_1}\,\d\theta^{(j)}\Big)^{2^{-(j-1)}}\\
&&= \Big(\int_{Y^{(j)}} \Big(\prod_{\eta \in \{0,1\}^j}\C^{|\eta|}f_d\circ \pi^{(j)}_\eta\Big)\,\d\nu^{(j)}\Big)^{2^{-(j-1)}},
\end{eqnarray*}
where $\bfZ^{(j)}$ is the auxiliary $H_{\{d-j,d\}}$-space constructed along with $\bfY^{(j)}$. Taking square-roots, this continues the induction.
\end{proof}

\subsection{Partially characteristic subspaces and the proof of convergence}\label{subs:part-char}

\begin{dfn}
Consider a probability space $(X,\mu)$, and a sequence $\Xi_n$ of multi-linear forms on $L^\infty(\mu)$ which are separately continuous for the norm $\|\cdot\|_2$ in each entry.  A closed subspace $V \leq L^2(\mu)$ is \textbf{partially characteristic in position $i$} for the sequence $\Xi_n$ if
\[\big\|\Xi_n(f_1,\ldots,f_d) - \Xi_n(f_1,\ldots,f_{i-1},P^Vf_i,f_{i+1},\ldots,f_d)\big\|_2 \to 0\]
as $n \to \infty$ for all $f_1$, \ldots, $f_d \in L^\infty(\mu)$, where $P^V$ is the orthogonal projection onto $V$.
\end{dfn}

The following proposition will quickly lead to a proof of Theorem A.  In fact, it gives rather more than one needs for the proof of Theorem A, but that extra strength will be used during the proof of Theorem B.

\begin{prop}\label{prop:pleasant}
For $1 \leq i \leq j \leq d$, let $\sfF_{i,j}$ be the functorial $\s$-algebra
\[\S^{\sfF_{i,j}}_\bfX := \bigvee_{\ell = 0}^{i-1}\S_\bfX^{T_{(\ell;i]}}\vee \bigvee_{\ell = i+1}^j\S_\bfX^{T_{(i;\ell]}},\]
and let $\sfV_{i,j,\bfX} := L^2(\mu|\S^{\sfF_{i,j}}_\bfX)$ be the associated functorial $L^2$-subspace.  Also, let
\[\hat{\L}_n^{(j)}(f_1,\ldots,f_j) := \frac{1}{|F_n|}\sum_{g \in F_n}\prod_{i=1}^{j}(f_i\circ T_{[1;i]}^g).\]

If $\bfX$ is $\sfV_{i,j}$-sated whenever $1 \leq i\leq j \leq d$, then, for each $j\in [d]$, the subspaces
\[\sfV_{1,j,\bfX},\ \ldots,\ \sfV_{j,j,\bfX}\]
are partially characteristic in positions $1$, \ldots, $j$ for the averages $\hat{\L}_n^{(j)}$.
\end{prop}

Notice that we still have $\hat{\L}_n^{(d)} = \L_n$ (the averages in~(\ref{eq:Lambda})), but otherwise these averages differ from the averages $\L_n^{(j)}$ considered in Theorem~\ref{thm:ineq}.

\begin{proof}
This is proved by induction on $j$.  When $j=1$, hence also $i=1$, we have $\S^{\sfF_{i,j}}_\bfX = \S^{T_1}_\bfX$, and this is always partially characteristic because the Norm Ergodic Theorem gives
\[\hat{\Lambda}^{(1)}_n(f_1) \to \sfE_\mu(f_1\,|\,\S^{T_1}_\bfX) \quad \hbox{in}\ \|\cdot\|_2\]
for any $G$-space.  So now we focus on the recursion clause.  For this it clearly suffices to assume $j=d-1$, and prove the result for the averages $\hat{\Lambda}^{(d)}_n$, which will lighten the notation.

\vspace{7pt}

\emph{Step 1.}\quad We first show that $\sfV_{d,d}$ is partially characteristic in position $d$. Let $f_d \in L^\infty(\mu)$. By decomposing it as
\[P^{\sfV_{d,d}}_\bfX f_d + (f_d - P^{\sfV_{d,d}}_\bfX f_d),\]
and using the multi-linearity of $\hat{\L}^{(d)}_n$, it suffices to show that
\[P^{\sfV_{d,d}}_\bfX f_d = 0 \quad \Longrightarrow \quad \|\hat{\L}^{(d)}_n(f_1,\ldots,f_d)\|_2 \to 0 \quad \forall f_1,\ldots,f_{d-1} \in L^\infty(\mu).\]

We will prove this in its contrapositive, so suppose that
\[\limsup_{n\to\infty}\|\hat{\L}^{(d)}_n(f_1,\ldots,f_d)\|_2 > 0 \quad \hbox{for some}\ f_1,\ldots,f_{d-1}.\]
Then Theorem~\ref{thm:ineq} gives
\[\int_{Y^{(d)}}\Big(\prod_{\eta \in \{0,1\}^d}\C^{|\eta|}f_d\circ \pi_\eta^{(d)}\Big)\,\d\nu^{(d)} \neq 0.\]
However, recalling relation~(\ref{eq:comm6}) from Lemma~\ref{lem:intertwine}, we see that if $\eta \in \{0,1\}^d\setminus \{0^d\}$ and $\ell \in [d]$ is maximal such that $\eta_\ell \neq 0$, then
\[\C^{|\eta|}f_d\circ \pi_\eta^{(d)}\circ S_{(d-\ell;d]}^g = \C^{|\eta|}f_d\circ \pi_\eta^{(d)} \quad \forall g \in G,\]
and so the function
\[\prod_{\eta \in \{0,1\}^d\setminus \{0^d\}}\C^{|\eta|}f_d\circ \pi^{(d)}_\eta\]
is measurable with respect to
\[\bigvee_{\ell=1}^d \S_{\bfY^{(d)}}^{S^{(d)}_{(d-\ell;d]}} = \S_{\bfY^{(d)}}^{\sfF_{d,d}}.\]

Therefore the non-vanishing of the above integral implies that
\[P^{\sfV_{d,d}}_{\bfY^{(d)}} (f_d\circ \pi_{0^d}^{(d)}) \neq 0.\]
Since $\bfX$ is $\sfV_{d,d}$-sated, this implies that also $P^{\sfV_{d,d}}_\bfX f_d \neq 0$, as required. This proves that the required subspace is partially characteristic for $i=j=d$.

\vspace{7pt}

\emph{Step 2.}\quad By Step 1, for any $f_1$, \ldots, $f_d \in L^\infty(\mu)$, we have
\[\hat{\L}_n^{(d)}(f_1,\ldots,f_d) - \hat{\L}_n^{(d)}(f_1,\ldots,f_{d-1},P^{\bfV_{d,d}}_\bfX f_d) \to 0.\]
Also, $P^{\bfV_{d,d}}_\bfX f_d$ still lies in $L^\infty(\mu)$, because $P^{\bfV_{d,d}}_\bfX$ is actually a conditional expectation operator. It therefore suffices to check that the required factors are partially characteristic in the other positions under the additional assumption that $f_d$ is $\S^{\sfF_{d,d}}_\bfX$-measurable.

This assumption implies that $f_d$ may be approximated in $\|\cdot\|_2$ by a finite sum of products of the form
\begin{eqnarray}\label{eq:product-decomp}
h_0\cdot \cdots \cdot h_{d-1}
\end{eqnarray}
where $h_i$ is $\S_\bfX^{T_{(i;d]}}$-measurable for each $i$.  By multi-linearity, it therefore suffices to prove that the required factors are partially characteristic in the other positions when $f_d$ is just one such product function.  However, at this point a simple re-arrangement gives
\begin{eqnarray*}
\hat{\L}_n^{(d)}(f_1,\ldots,f_{d-1},h_0\cdots h_{d-1}) &=& \frac{1}{|F_n|}\sum_{g \in F_n}\Big(\prod_{i=1}^{d-1}(f_i\circ T_{[1;i]}^g)\Big)\cdot ((h_0\cdots h_{d-1})\circ T_{[1;d]}^g)\\
&=& h_0\cdot\frac{1}{|F_n|}\sum_{g \in F_n}\prod_{i=1}^{d-1}((f_ih_i)\circ T_{[1;i]}^g)\\
&=& h_0\cdot \hat{\L}_n^{(d-1)}(f_1h_1,\ldots,f_{d-1}h_{d-1}),
\end{eqnarray*}
using the partial invariances of each of the $h_i$s. Therefore, by the inductive hypothesis for $j = d-1$, if these averages do not vanish as $n \to\infty$ then
\[P^{\sfV_{i,d-1}}_\bfX(f_ih_i) = \sfE_\mu\big(f_ih_i\,\big|\,\S^{\sfF_{i,d-1}}_\bfX\big)\neq 0\]
for each $i \in [d-1]$.  Since $h_i$ is $\S_\bfX^{T_{(i;d]}}$-measurable, this implies that
\[\sfE_\mu\big(f_i\,\big|\,\S^{\sfF_{i,d-1}}_\bfX\vee \S^{T_{(i;d]}}_\bfX\big) = \sfE_\mu(f_i\,|\,\S^{\sfF_{i,d}}_\bfX ) \neq 0.\]
Arguing as at the start of Step 1, this implies that for the averages $\hat{\L}_n^{(d)}$, the subspace $\sfV_{i,d,\bfX}$ is partially characteristic in position $i$ for each $i\leq d-1$. Therefore the induction continues.
\end{proof}

\begin{proof}[Proof of Theorem A]
The proof is by induction on $d$, and uses only the partially characteristic factor in position $d$.  When $d=1$, convergence is given by the Norm Ergodic Theorem for amenable groups, so suppose $d\geq 2$, and that convergence is known for all commuting tuples of fewer than $d$ actions.

By Theorem~\ref{thm:sateds-exist}, we may ascend from $\bfX$ to an extension which is $\sfV_{d,d}$-sated, and so simply assume that $\bfX$ is itself $\sfV_{d,d}$-sated.  This implies that
\[\big\|\L_n(f_1,\ldots,f_d) - \L_n(f_1,\ldots,f_{d-1},P^{\sfV_{d,d}}_\bfX f_d)\big\|_2 \to 0\quad \hbox{as}\ n\to\infty,\]
as in the proof of the previous proposition.  It therefore suffices to prove convergence for the right-hand averages inside these norms.  However, again as in the proof of the previous proposition, for these we may approximate $\sfE_\mu(f_d\,|\,\S_\bfX^{\sfF_{d,d}})$ by a finite sum of products of the form in~(\ref{eq:product-decomp}), and then re-arrange the resulting averages into the form
\[h_0\cdot\frac{1}{|F_n|}\sum_{g \in F_n}\prod_{i=1}^{d-1}((f_ih_i)\circ T_{[1;i]}^g).\]
Ignoring the factor $h_0$, which is uniformly bounded and does not depend on $n$, this is now a system of non-conventional averages for a commuting $(d-1)$-tuple of $G$-actions, so convergence follows by the inductive hypothesis.
\end{proof}

\section{Proof of multiple recurrence}

Given a $G^d$-space, the convergence proved in Theorem A implies that the sequence of measures
\[\l_n := \frac{1}{|F_n|}\sum_{g \in F_n} \delta_{(T_1^gx,T_{[1;2]}^gx,\ldots,T_{[1;d]}^gx)}\]
converges in the usual topology on the convex set of $d$-fold couplings of $\mu$.  (This is the same as the topology on joinings when one gives all spaces the action of the trivial group; the joining topology is explained, for instance, in~\cite[Section 6.1]{Gla03}.)

Letting $\l := \lim_{n\to\infty}\l_n$, it follows that for any measurable $A \subseteq X$ one has
\[\frac{1}{|F_n|}\sum_{g \in F_n}\mu(T_1^gA\cap \cdots \cap T_{[1;d]}^gA) \to \l(A^d).\]
To complete the proof of Theorem B, we will show that
\begin{eqnarray}\label{eq:multirec}
\l(A_1\times \cdots \times A_d) = 0 \quad \Longrightarrow \quad \mu(A_1\cap \cdots \cap A_d) = 0
\end{eqnarray}
for any $G^d$-space $\bfX = (X,\mu,T)$ and measurable subsets $A_1$, \ldots, $A_d \subseteq X$.  This gives the desired conclusion by setting $A_1 := \ldots := A_d := A$.

First, by replacing $\bfX$ with a suitable extension and lifting each $A_i$ to that extension, we may reduce this task to the case in which $\bfX$ is sated with respect to any chosen family of functional $\s$-subalgebras.  After doing this, the result will be proved by making contact with a modification of Tao's Infinitary Removal Lemma from~\cite{Tao06--hyperreg}.  This is the same strategy as in~\cite{Aus--newmultiSzem}. The modification of the Removal Lemma is essentially as in~\cite{Aus--newmultiSzem}, but we shall use its more explicit formulation from~\cite{Aus--thesis}:

\begin{prop}[{\cite[Proposition 4.3.1]{Aus--thesis}}]\label{prop:infremove}
Let $(X,\mu)$ be a standard Borel probability space with $\s$-algebra $\S$.  Let $\pi_i:X^d \to X$ be the coordinate projection for each $i\leq d$.   Let $\theta$ be a $d$-fold coupling of $\mu$ on $X^d$.  Finally, suppose that $(\Psi_e)_e$ is a collection of $\s$-subalgebras of $\S$, indexed by $e \in \binom{[d]}{\geq 2}$, with the following properties:
\begin{itemize}
\item[i)] if $a \subseteq e$ then $\Psi_a \supseteq \Psi_e$;
\item[ii)] if $i,j \in e$ and $A \in \Psi_e$ then $\theta(\pi_i^{-1}(A)\triangle \pi_j^{-1}(A)) = 0$, so that we may let $\hat{\Psi}_e$ denote the common $\theta$-completion of the lifted $\s$-algebras $\pi_i^{-1}(\Psi_e)$ for $i \in e$;
\item[iii)] if we define $\hat{\Psi}_\I := \bigvee_{e \in \I}\hat{\Psi}_e$ for each up-set $\I \subseteq \binom{[d]}{\geq 2}$, then the $\s$-algebras $\hat{\Psi}_\I$ and $\hat{\Psi}_{\J}$ are relatively independent under $\theta$ over $\hat{\Psi}_{\I\cap \J}$.
\end{itemize}
In addition, suppose that $\I_{i,j}$ for $i=1,2,\ldots,d$ and $j = 1,2,\ldots,k_i$ are up-sets in $\binom{[d]}{\geq 2}$ such that $[d] \in \I_{i,j}\subseteq \langle i\rangle$ for each $i,j$, and that $A_{i,j} \in \bigvee_{e \in \I_{i,j}}\Psi_e$ for each $i,j$.  Then
\[\theta\Big(\prod_{i=1}^d\Big(\bigcap_{j=1}^{k_i}A_{i,j}\Big)\Big) = 0 \quad \Longrightarrow \quad \mu\Big(\bigcap_{i=1}^d\bigcap_{j=1}^{k_i}A_{i,j}\Big) = 0.\]
\end{prop}

This result is proved by a rather lengthy induction on the up-sets $\I_{i,j}$, which requires the full generality above: see~\cite[Subsection 4.3]{Aus--thesis} for a proof and additional discussion.  However, as in~\cite{Aus--newmultiSzem}, we will apply it only for $k_i = 1$ and $\I_{i,1} = \langle i \rangle$ for each $i$, in which case it asserts that, if $A_i \in \Psi_{\langle i\rangle}$ for each $i$, then
\begin{eqnarray}\label{eq:spec-case}
\theta(A_1\times\cdots \times A_d) = 0 \quad \Longrightarrow \quad \mu(A_1\cap \cdots \cap A_d) = 0.
\end{eqnarray}

We will apply Proposition~\ref{prop:infremove} with $\theta$ equal to the limit coupling $\l$, and with the following family of $\s$-subalgebras.  Suppose that $e = \{i_1 < \ldots < i_\ell\}\subseteq [d]$ is non-empty, and define a new functorial $\s$-subalgebra of $G^d$-spaces $\bfX$ by
\begin{multline*}
\Phi^e_\bfX := \S_\bfX^{L_e} = \big\{A \in \S_\bfX\,\big|\ \mu(T^g_{(i_1;i_2]}A \triangle A) = \mu(T^g_{(i_2;i_3]}A\triangle A) =\\ \ldots = \mu(T^g_{(i_{\ell-1};i_\ell]}A\triangle A) = 0\ \forall g\in G\big\},
\end{multline*}
where we interpret this as $\S_\bfX$ in case $|e| = 1$.  Observe that if $a \subseteq e$, then $L_a \leq L_e$, and so $\Phi^a_\bfX \supseteq \Phi^e_\bfX$.  For any up-set $\I \subseteq \binom{[d]}{\geq 2}$, let
\[\Phi^\I_\bfX := \bigvee_{e \in \I}\Phi^e_\bfX.\]

Most of our remaining work will go into checking properties (i)--(iii) above for the joint distribution of the lifted $\s$-algebras $\pi_i^{-1}(\Phi_\bfX^e)$, subject to a certain satedness assumption on $\bfX$.

\begin{dfn}
The $G^d$-space $\bfX = (X,\mu,T)$ is \textbf{fully sated} if it is sated for every functorial $\s$-subalgebra of the form
\[\bigvee_{s=1}^r\Phi^{e_s}_\bfX\]
for some $e_1$, \ldots, $e_r \in \binom{[d]}{\geq 2}$.
\end{dfn}

\begin{thm}\label{thm:struct-of-coupling}
Let $\bfX$ be fully sated, and let $\l$ be its limit coupling as above.
\begin{enumerate}
\item[(1)] The coordinates factors $\pi_i:X^d\to X$ are relatively independent under $\l$ over the further $\s$-subalgebras $\pi_i^{-1}(\Phi^{\langle i\rangle}_\bfX)$, $i=1,2,\ldots,d$.
\item[(2)] The collection $(\Phi^e_\bfX)_e$ satisfies properties (i)--(iii) of Proposition~\ref{prop:infremove}.
\end{enumerate}
\end{thm}

\begin{proof}[Proof of Theorem B from Theorem~\ref{thm:struct-of-coupling}]
By passing to an extension as given by Corollary~\ref{cor:mult-sateds-exist}, it suffices to prove Theorem B for fully sated $G^d$-spaces.  Specifically, we will prove the implication~(\ref{eq:multirec}).

Let $A_1$, \ldots, $A_d \subseteq X$ be measurable.  Part (1) of Theorem~\ref{thm:struct-of-coupling} gives
\[\l(A_1\times \cdots \times A_d) = \int_X f_1\otimes \cdots \otimes f_d \,\d\l,\]
where $f_i := \sfE_\mu(1_{A_i}\,|\,\Phi_\bfX^{\langle i\rangle})$.  Letting $B_i := \{f_i > 0\}$ for each $i$, it follows that
\begin{multline*}
\l(A_1\times \cdots \times A_d) = 0 \\ \Longrightarrow \quad \l\{f_1\otimes \cdots \otimes f_d > 0\} = \l(B_1\times \cdots \times B_d) = 0.
\end{multline*}

On the other hand, each $B_i$ is $\Phi^{\langle i\rangle}_\bfX$-measurable.  By Part (2) of Theorem~\ref{thm:struct-of-coupling}, we may therefore apply Proposition~\ref{prop:infremove} in the form of the implication~(\ref{eq:spec-case}), to conclude
\[\l(B_1\times \cdots \times B_d) = 0 \quad \Longrightarrow \quad \mu(B_1\cap \cdots \cap B_d) = 0.\]
Since
\[\mu(A_i\setminus B_i) = \int 1_{A_i}\cdot 1_{X\setminus B_i}\,\d\mu  = \int f_i\cdot 1_{X\setminus B_i}\,\d\mu = 0\]
for each $i$, this completes the proof.
\end{proof}

The rest of this section will go into proving Theorem~\ref{thm:struct-of-coupling}.  Subsection~\ref{subs:jointdist1} will establish various necessary joint-distribution properties of the $\s$-algebras $\Phi^e_\bfX$ in $(X,\mu)$ itself, and then Subsection~\ref{subs:jointdist2} will deduce the required properties of the lifts $\pi_i^{-1}(\Phi^e_\bfX)$ from these.

\subsection{Joint distribution of some $\s$-algebras of invariant sets}\label{subs:jointdist1}

The next proposition will be the second major outing for satedness in this paper.  It shows that if a $G^d$-space $\bfX$ is sated relative to a suitable family of functorial $\s$-subalgebras constructed out of the collection $\Phi^e_{\bfX}$, $e\subseteq [k]$, then this forces some relative independence among those $\s$-subalgebras.

\begin{prop}\label{prop:first-rel-ind}
Suppose that $\{i < j\} \subseteq [d]$, suppose that $e_1$, \ldots, $e_r \in \binom{[d]}{\geq 2}$, and let $\bfX = (X,\mu,T_1,\ldots,T_d)$ be a $G^d$-space.

\begin{enumerate}
\item[(1)] If $e_s\cap [i;j) = \{i\}$ for every $s \leq r$, and if $\bfX$ is $\sfF$-sated for
\[\S^\sfF_\bfX := \bigvee_{s=1}^r\Phi_{\bfX}^{e_s\cup \{j\}},\]
then
\[\Phi_\bfX^{\{i,j\}} \quad \hbox{and} \quad \bigvee_{s=1}^r \Phi^{e_s}_\bfX \quad \hbox{are rel. indep. over} \quad \S^\sfF_\bfX.\]

\item[(2)] If $e_s\cap (i;j] = \{j\}$ for every $s \leq r$, and if $\bfX$ is $\sfG$-sated for
\[\S^\sfG_{\bfX} := \bigvee_{s=1}^r\Phi_{\bfX}^{e_s\cup \{i\}},\]
then
\[\Phi_\bfX^{\{i,j\}} \quad \hbox{and} \quad \bigvee_{s=1}^r \Phi^{e_s}_\bfX \quad \hbox{are rel. indep. over} \quad \S^\sfG_\bfX.\]
\end{enumerate}
\end{prop}

\begin{proof}[Proof of Part (1).]
As always, the appeal to satedness will depend on constructing the right extension.  Let $e := e_1 \cup \ldots \cup e_r \cup \{j\}$, so that our assumptions give $e\cap [i;j] = \{i,j\}$.

\vspace{7pt}

\emph{Step 1.}\quad We first construct a suitable extension of the $L_e$-subaction $\bfX^{\uhr L_e}$.

Since $e \cap [i;j] = \{i,j\}$, Corollary~\ref{cor:some-invce} promises that $\Phi_\bfX^{\{i,j\}}$ is globally $L_e$-invariant, so the relative product measure $\nu:=\mu\otimes_{\Phi_\bfX^{\{i,j\}}}\mu$ is invariant under the diagonal action of $L_e$ on $Y := X^2$.  We will make use of this by constructing a \emph{non}-diagonal action of $L_e$ on $Y$ which still preserves $\nu$.

Let $e$ be enumerated as $\{i_1 < \ldots < i_m\}$.  Our assumptions imply that $\{i,j\} = \{i_{\ell_0},i_{\ell_0+1}\}$ for some $\ell_0 \leq m-1$.  In these terms, $L_e$ is generated by its $m-1$ commuting subgroups
\[L_{\{i_\ell,i_{\ell+1}\}} = \phi_\ell(G),\quad \ell = 1,\ldots,m-1,\]
where $\phi_\ell:G\to G^d$ is the injective homomorphism defined by
\[\big(\phi_\ell(g)\big)_i := \left\{\begin{array}{ll}g &\quad \hbox{if}\ i \in (i_\ell;i_{\ell+1}] \\ e &\quad \hbox{else}.\end{array}\right.\]

Specifying an action of $L_e$ is equivalent to specifying commuting actions of its subgroups $\phi_\ell(G)$.  We define our new, non-diagonal action $S:L_e\actson (Y,\nu)$ as follows:
\[S^{\phi_\ell(g)} := \left\{\begin{array}{ll}T_{(i_\ell;i_{\ell+1}]}^g\times T_{(i_\ell,i_{\ell+1}]}^g \quad \hbox{(i.e., diagonal)} &\quad \hbox{if}\ \ell \not\in \{\ell_0,\ell_0+1\}\\
T_{(i_{\ell_0};i_{\ell_0+1}]}^g\times \rm{id} &\quad \hbox{if}\ \ell = \ell_0\\
T_{(i_{\ell_0+1};i_{\ell_0+2}]}^g\times T_{(i_{\ell_0},i_{\ell_0+2}]}^g &\quad \hbox{if}\ \ell = \ell_0+1\end{array}\right.\]
(where the last option here is vacuous in case $\ell_0 = m-1$).

Each of these transformations leaves $\nu$ invariant:  we have already remarked this for the diagonal transformations, and for the last two possibilities we need only observe that $\nu$ is a relative product over a $\s$-algebra on which $T^g_{(i_{\ell_0};i_{\ell_0+1}]} = T^g_{(i;j]}$ acts trivially for all $g \in G$.

Now let $\b_1,\b_2:Y\to X$ be the two coordinate projections.  The above definition gives $\b_1\circ S^{\phi_\ell(g)} = T^{\phi_\ell(g)}\circ \b_1$ for every $\ell$ and $g$, so letting $\bfY$ denote the $L_e$-space given by the above transformations $S$ on $(Y,\nu)$, we have a factor map $\bfY\stackrel{\b_1}{\to} \bfX^{\uhr L_e}$.  On the other hand, recalling that $\{i_{\ell_0},i_{\ell_0+1}\} = \{i,j\}$, the above definitions give that under $S$, the subgroup $L_{\{i,j\}} \leq L_e$ acts trivially on the second coordinate in $Y$.

\vspace{7pt}

\emph{Step 2.}\quad Next, applying Theorem~\ref{thm:recoverG} enlarges this to an extension $\bfX_1\stackrel{\pi}{\to} \bfX$ of $G^d$-spaces which factorizes through some $L_e$-extension $\bfX_1^{\uhr L_e}\stackrel{\a}{\to} \bfY$.

\vspace{7pt}

\emph{Step 3.}\quad Now suppose that $f \in L^\infty(\mu|_{\Phi_\bfX^{\{i,j\}}})$ and that $g_s \in L^\infty(\mu|_{\Phi^{e_s}_\bfX})$ for each $s \leq r$.  From the definition of $\nu$ and the fact that $f$ is $\Phi^{\{i,j\}}_\bfX$-measurable, we have
\begin{multline}\label{eq:move-int}
\int_X f\cdot \prod_{s \leq r}g_s\,\d\mu = \int_Y (f\circ \b_2)\cdot \Big(\prod_{s \leq r}g_s\circ \b_2\Big)\,\d\nu\\ = \int_Y (f\circ\b_1)\cdot \Big(\prod_{s\leq r}g_s\circ \b_2\Big)\,\d\nu.
\end{multline}

Consider the function $g_s\circ \b_2$ for some $s\leq r$.  Our next step is to show that it is invariant under the whole of $S^{\uhr L_{e_s\cup \{j\}}}$.  For a given $s$, if we enumerate $e_s \cup \{j\} =: \{p_1 < \ldots < p_n\}$, then it suffices to prove invariance under each subgroup $L_{\{p_k,p_{k+1}\}}$ for $k \in \{1,2,\ldots,n-1\}$.

There are three cases to consider.  Firstly, if $p_k \not\in \{i,j\}$, then, since $e\cap [i;j) = \{i\}$, it follows that $(p_k;p_{k+1}]$ is disjoint from $(i_{\ell_0};i_{\ell_0+2}]$ (using again the notation from Step 1).  In this case the definition of $S$ gives $\b_2 \circ S_{(p_k;p_{k+1}]} = T_{(p_k;p_{k+1}]}\circ \b_2$, so the required invariance follows from the fact that $g_s$ itself is $L_{e_s}$-invariant.

Second, if $p_k = i$, then the assumption $e_s \cap [i;j) = \{i\}$ implies that $p_{k+1} = j$, and hence the definition of $S$ implies that $\b_2 \circ S_{(p_k;p_{k+1}]} = \b_2$, from which the $S_{(p_k;p_{k+1}]}$-invariance of $g_s\circ \b_2$ is obvious.

Finally, if $p_k = j$, then the definition of $S$ gives
\[\b_2\circ S_{(p_k;p_{k+1}]} = T_{(i;p_{k+1}]}\circ \b_2.\]
Since $\{i,p_{k+1}\} \subseteq e_s$ (even if $j\not\in e_s$), once again the $L_{e_s}$-invariance of $g_s$ gives
\[g_s\circ \b_2\circ S^g_{(p_k;p_{k+1}]} = g_s\circ T^g_{(i;p_{k+1}]}\circ \b_2 = g_s\circ \b_2,\]
as required.

\vspace{7pt}

\emph{Step 4.}\quad In light of Step 3, the function $\prod_{s \leq r}(g_s\circ \b_2\circ \a)$ is measurable with respect to
\[\bigvee_{s\leq r}\Phi_{\bfX_1}^{e_s\cup \{j\}} = \S^\sfF_{\bfX_1}.\]
By the assumed $\sfF$-satedness, it follows that the right-hand integral in~(\ref{eq:move-int}) is equal to
\[\int_Y (\sfE_\mu(f\,|\ \S_\bfX^\sfF)\circ\b_1)\cdot \Big(\prod_{s\leq r}g_s\circ \b_2\Big)\,\d\nu,\]
and by the same reasoning that gave~(\ref{eq:move-int}) itself this is equal to
\[\int_X \sfE_\mu(f\,|\ \S_\bfX^\sfF)\cdot \prod_{s \leq r}g_s\,\d\mu = \int_X \sfE_\mu(f\,|\ \S_\bfX^\sfF)\cdot \sfE_\mu\Big(\prod_{s \leq r}g_s\,\Big|\,\S^\sfF_\bfX\Big)\,\d\mu.\]

Since $f$ and each $g_s$ were arbitrary subject to their measurability assumptions, this implies that $\Phi_\bfX^{\{i,j\}}$ and $\bigvee_{s=1}^r\Phi_\bfX^{e_s}$ are relatively independent over $\S^\sfF_\bfX$.
\end{proof}

\begin{proof}[Proof of Part (2).]
This follows exactly the same steps as Part 1, except that now the new $L_e$-action $S$ on
\[(Y,\nu) := (X^2,\mu\otimes_{\Phi^{\{i,j\}}_\bfX}\mu)\]
is defined as follows:
\[S^{\phi_\ell(g)} := \left\{\begin{array}{ll}T_{(i_\ell;i_{\ell+1}]}^g\times T_{(i_\ell,i_{\ell+1}]}^g \quad \hbox{(i.e., diagonal)} &\quad \hbox{if}\ \ell \not\in \{\ell_0-1,\ell_0\}\\
T_{(i_{\ell_0};i_{\ell_0+1}]}^g\times \rm{id} &\quad \hbox{if}\ \ell = \ell_0\\
T_{(i_{\ell_0-1};i_{\ell_0}]}^g\times T_{(i_{\ell_0-1},i_{\ell_0+1}]}^g &\quad \hbox{if}\ \ell = \ell_0-1,\end{array}\right.\]
where now $e := e_1\cup \cdots \cup e_r\cup \{i\} = \{i_1 < \ldots < i_m\}$, and $\ell_0$ is such that $\{i,j\} = \{i_{\ell_0},i_{\ell_0 + 1}\}$, as before.
\end{proof}

The next two propositions contain the consequences of full satedness that we need.  The first modifies the conclusion of Proposition~\ref{prop:pleasant} in the case of integrated averages, rather than functional averages.

\begin{prop}\label{prop:char-factors-again}
Suppose that $\bfX$ is fully sated, that $\{i_1 < \ldots < i_k\} \subseteq [d]$ and that $f_1$, \ldots, $f_k \in L^\infty(\mu)$.  Then
\begin{multline*}
\lim_{n\to\infty}\int_X \frac{1}{|F_n|}\sum_{g \in F_n}\prod_{j=1}^k (f_j\circ T_{[1;i_j]}^g)\,\d\mu\\ = \lim_{n\to\infty}\int_X \frac{1}{|F_n|}\sum_{g \in F_n}\prod_{j=1}^k (\sfE_\mu(f_j\,|\,\Delta_j)\circ T_{[1;i_j]}^g)\,\d\mu,
\end{multline*}
where
\[\Delta_j := \bigvee_{\ell = 1}^{j-1}\S^{T_{(i_\ell;i_j]}}_\bfX \vee \bigvee_{\ell = j+1}^k\S^{T_{(i_j;i_\ell]}}_\bfX.\]
\end{prop}

\begin{proof}
This is proved by induction on $k$.  When $k=1$ it is trivial, by the $T_{[1;i_1]}$-invariance of $\mu$, so suppose $k\geq 2$.

Because $\mu$ is $T_{[1;i_1]}$-invariant, the desired conclusion is equivalent to
\begin{multline*}
\int_X f_1\cdot \Big(\lim_{n\to\infty}\frac{1}{|F_n|}\sum_{g \in F_n}\prod_{j=2}^k (f_j\circ T_{(i_1;i_j]}^g)\Big)\,\d\mu\\ = \int_X \sfE_\mu(f_1\,|\,\Delta_1)\cdot \Big(\lim_{n\to\infty}\frac{1}{|F_n|}\sum_{g \in F_n}\prod_{j=2}^k (\sfE_\mu(f_j\,|\,\Delta_j)\circ T_{(i_1;i_j]}^g)\Big)\,\d\mu.
\end{multline*}
However, $\bfX$ being fully sated implies that the $G^{k-1}$-space $\bfX'$ defined by
\[T'_j := T_{(i_j,i_{j+1}]} \quad \hbox{for}\ j=1,2,\ldots,k-1\]
is also fully sated: otherwise, we could turn a $G^{k-1}$-extension witnessing the failure of satedness for $\bfX'$ back into a $G^d$-extension of $\bfX$ using Theorem~\ref{thm:recoverG}.  Therefore Proposition~\ref{prop:pleasant} applied to this $G^d$-space gives
\begin{multline*}
\int_X f_1\cdot \Big(\lim_{n\to\infty}\frac{1}{|F_n|}\sum_{g \in F_n}\prod_{j=2}^k (f_j\circ T_{(i_1;i_j]}^g)\Big)\,\d\mu\\ = \int_X f_1\cdot \Big(\lim_{n\to\infty}\frac{1}{|F_n|}\sum_{g \in F_n}\prod_{j=2}^k (\sfE_\mu(f_j\,|\,\Delta_j)\circ T_{(i_1;i_j]}^g)\Big)\,\d\mu,
\end{multline*}
since the $\s$-algebras $\Delta_j$ for $j\geq 2$ are those that arise by applying the functorial $\s$-subalgebras $\sfF_{\bullet,\bullet}$ of that proposition to the $G^{k-1}$-space $\bfX'$.

This almost completes the proof. To finish, observe that, as in the proof of Theorem A, we may now approximate $f_k$ (say) by a finite sum of finite products of functions measurable with respect to $\S_\bfX^{T_{(i_\ell,i_k]}}$ for $\ell=1,\ldots,k-1$, and having done so we may re-arrange the above into an analogous system of averages with only $k-1$ transformations.  At that point, the inductive hypothesis allows us to replace $f_1$ with $\sfE_\mu(f_1\,|\,\Delta_1)$, completing the proof.
\end{proof}

\begin{prop}\label{prop:next-rel-ind}
Suppose that $\bfX$ is fully sated, and that $e_0$, $e_1$, \ldots, $e_r \in \binom{[d]}{\geq 2}$ are sets such that $\bigcap_{0 \leq s\leq r}e_s \neq \emptyset$.  Let $g$ be $\Phi^{e_0}_\bfX$-measurable and $f_s$ be $\Phi^{e_s}_\bfX$-measurable for $1 \leq s \leq r$, and suppose $g$ and every $f_s$ are bounded.  Then
\[\int_X g\cdot f_1\cdot\cdots \cdot f_r\,\d\mu = \int_X\sfE_\mu(g\,|\,\Psi)\cdot f_1\cdot \cdots \cdot f_r\,\d\mu,\]
where
\[\Psi:= \bigvee_{s=1}^r \Phi^{e_0\cup e_s}_\bfX.\]
\end{prop}

As is standard, this is equivalent to the assertion that
\[\sfE_\mu\Big(g\,\Big|\,\bigvee_{s=1}^r\Phi^{e_s}_\bfX\Big) = \sfE_\mu(g\,|\,\Psi).\]

\begin{proof}
If $e_0 = e_1 = e_2 = \cdots = e_r$, then $\Phi^{e_s}_\bfX = \Psi$ for all $s$, so the result is trivial.

The general case is proved by an outer induction on $r$, and an inner induction on the cardinality of the up-set
\[\cal{A}_{e_0,e_1,\ldots,e_r} := \langle e_0\rangle \cup \cdots \cup \langle e_r\rangle.\]

The base clause corresponds to $r=1$ and $|\cal{A}_{e_0,e_1}| = 1$, hence $e_0 = e_1 = [k]$: this is among the trivial cases described above.

For the recursion clause, we may assume that there are at least two distinct sets among the $e_s$ for $0\leq s \leq r$, so $\bigcap_{0 \leq s \leq r}e_s \neq \bigcup_{0 \leq s \leq r}e_s$ (else we would be in the trivial case treated above).  We have also assumed that $\bigcap_{0 \leq  s \leq r}e_s \neq \emptyset$, so there must be $\{i < j\} \subseteq [k]$ such that $e_s \cap [i+1;j) = \emptyset$ for all $0 \leq s\leq r$, and
\begin{itemize}
\item either $i \in e_s$ for all $0 \leq s \leq r$, but $j$ lies in some $e_s$ but not all;
\item or $j \in e_s$ for all $0 \leq s \leq r$, but $i$ lies in some $e_s$ but not all.
\end{itemize}

We will complete the induction in the first of these cases, the second being exactly analogous.  In this first case we have $e_s\cap [i;j) = \{i\}$ for all $s$.  It now breaks into two further sub-cases.

\vspace{7pt}

\emph{Case 1.}\quad First assume that $j \in e_0$.  Then, since $\Phi^{e_0}_\bfX \subseteq \Phi^{\{i,j\}}_\bfX$, Part 1 of Proposition~\ref{prop:first-rel-ind} gives that
\[\int_X g\cdot (f_1\cdot \cdots \cdot f_r)\,\d\mu = \int_X\sfE_\mu\Big(g\,\Big|\,\bigvee_{s=1}^r\Phi^{e_s\cup \{j\}}_\bfX\Big)\cdot (f_1\cdot \cdots \cdot f_r)\,\d\mu.\]

Now observe that
\[\cal{A}_{e_0,e_1\cup\{j\},\ldots,e_r\cup \{j\}} \subseteq \cal{A}_{e_0,\ldots,e_r},\]
and this inclusion is strict, because the right-hand family contains some set that does not contain $j$, whereas every element of the left-hand family does contain $j$.  Therefore we may apply the inductive hypothesis to $e_0$ and $e_1\cup \{j\}$, $e_2\cup \{j\}$, \ldots, $e_r\cup\{j\}$, to conclude that
\[\sfE_\mu\Big(g\,\Big|\,\bigvee_{s=1}^r\Phi^{e_s\cup \{j\}}_\bfX\Big) = \sfE_\mu\Big(g\,\Big|\,\bigvee_{s=1}^r\Phi^{e_0 \cup e_s}_\bfX\Big) = \sfE_\mu(g\,|\,\Psi).\]

\vspace{7pt}

\emph{Case 2.}\quad Now assume that $j \not\in e_0$.  By re-labeling the other sets if necessary, we may assume $j \in e_1$. Then the argument used in Case 1 gives
\[\int_X g\cdot f_1\cdot\cdots \cdot f_r\,\d\mu = \int_X g\cdot \sfE_\mu(f_1\,|\,\Psi_1)\cdot f_2 \cdot \cdots \cdot f_r\,\d\mu,\]
where
\[\Psi_1:= \bigvee_{s\in\{0\}\cup [r]\setminus \{1\}}\Phi^{e_1\cup e_s}_\bfX,\]
and similarly
\[\int_X \sfE_\mu(g\,|\,\Psi)\cdot f_1\cdot\cdots \cdot f_r\,\d\mu = \int_X \sfE_\mu(g\,|\,\Psi)\cdot \sfE_\mu(f_1\,|\,\Psi_1)\cdot f_2\cdot \cdots \cdot f_r\,\d\mu.\]
It therefore suffices to prove the desired equality when $f_1$ is $\Psi_1$-measurable.  Since such an $f_1$ can be approximated in $\|\cdot\|_2$ by finite sums of products of $\Phi^{e_1\cup e_s}_\bfX$-measurable functions for $s \in \{0\}\cup [r]\setminus \{1\}$, it suffices furthermore to assume that
\[f_1 = f_{10}\cdot f_{12}\cdot \cdots \cdot f_{1r}\]
is one such product.

However, now the integral of interest may be written as
\[\int_X (g f_{10})\cdot (f_2f_{12})\cdot \cdots \cdot (f_rf_{1r})\,\d\mu,\]
and by the inductive hypothesis on $r$ this is equal to
\[\int_X \sfE_\mu(g f_{10}\,|\,\Psi_2)\cdot (f_2f_{12})\cdot \cdots \cdot (f_rf_{1r})\,\d\mu\]
with
\[\Psi_2:= \bigvee_{s=2}^r\Phi^{e_0\cup e_s}_\bfX.\]

Since $\Psi \geq \Psi_2$ and $f_{10}$ is $\Psi$-measurable, the Law of Iterated Conditional Expectation gives
\[\sfE_\mu(g f_{10}\,|\,\Psi_2) = \sfE_\mu(\sfE_\mu(g\,|\,\Psi) f_{10}\,|\,\Psi_2).\]
Substituting this back into the integral completes the desired equality.
\end{proof}

\begin{cor}\label{cor:second-rel-ind}
Suppose that $\bfX$ is fully sated and that $e_1$, \ldots, $e_r \in \binom{[d]}{\geq 2}$ are such that $\bigcap_{s\leq r}e_s \neq \emptyset$.  Let
\[\Psi_s:= \bigvee_{t \in [r]\setminus s} \Phi^{e_t\cup e_s}_\bfX \quad \hbox{for}\ s=1,2,\ldots,r,\]
and let $f_s$ be $\Phi^{e_s}_\bfX$-measurable for each $s$.  Then
\[\int_X \prod_{s=1}^r f_s\ \d\mu = \int_X\prod_{s=1}^r\sfE_\mu(f_s\,|\,\Psi_s)\ \d\mu.\]
\end{cor}

\begin{proof}
Simply apply Proposition~\ref{prop:next-rel-ind} to each factor of the integrand in turn.
\end{proof}

\subsection{Structure of the limit couplings}\label{subs:jointdist2}

We now turn to the structure of the limit coupling $\l$, as discussed at the beginning of this section. The following lemma and definition are again essentially copied from~\cite{Aus--newmultiSzem}.

\begin{lem}\label{lem:diag}
If $i,j \in e \in \binom{[d]}{\geq 2}$ and $A \in \Phi^e_\bfX$, then
\[\l(\pi_i^{-1}(A)\triangle \pi_j^{-1}(A)) = 0.\]
In particular, $\pi_i^{-1}(\Phi^e_\bfX)$ and $\pi_j^{-1}(\Phi_\bfX^e)$ agree up to $\l$-negligible sets.
\end{lem}

\begin{proof}
If $i = j$ then this is trivial, so assume without loss of generality that $i < j$.  By definition,
\begin{eqnarray*}
\l(\pi_i^{-1}(A)\cap \pi_j^{-1}(A)) &=& \lim_{n\to\infty}\frac{1}{|F_n|}\sum_{g \in F_n} \mu(T_{[1;i]}^{g^{-1}}A \cap T_{[1;j]}^{g^{-1}}A)\\
&=& \lim_{n\to\infty}\frac{1}{|F_n|}\sum_{g \in F_n} \mu(T_{[1;i]}^{g^{-1}}A \cap T_{[1;i]}^{g^{-1}}T_{(i;j]}^{g^{-1}}A)\\
&=& \lim_{n\to\infty}\frac{1}{|F_n|}\sum_{g \in F_n} \mu(T_{[1;i]}^{g^{-1}}A \cap T_{[1;i]}^{g^{-1}}A) = \mu(A),
\end{eqnarray*}
because $\Phi^e_\bfX \leq \Phi^{\{i,j\}}_\bfX$, so $A$ is $T_{(i;j]}$-invariant.  Therefore
\[\l(\pi_i^{-1}(A)\cap \pi_j^{-1}(A)) = \l(\pi_i^{-1}(A)) = \l(\pi_j^{-1}(A)),\]
and so $\l(\pi_i^{-1}(A)\triangle \pi_j^{-1}(A)) = 0$.
\end{proof}

\begin{dfn}
In the setting above, the common $\l$-completion of the lifted $\s$-algebras
\[\pi_i^{-1}(\Phi^e_\bfX) \quad \hbox{for}\ i \in e\]
is called the \textbf{oblique copy} of $\Phi^e_\bfX$. It will be denoted by $\hat{\Phi}^e_\bfX$.
\end{dfn}

\begin{dfn}
If $\cal{I}$ is a non-empty up-set in $\binom{[d]}{\geq 2}$ and $\bfX$ is a $G^d$-space, then the $\cal{I}$-oblique $\s$-algebra is
\[\hat{\Phi}^{\cal{I}}_\bfX := \bigvee_{e \in \cal{I}}\hat{\Phi}^e_\bfX.\]
\end{dfn}

The remainder of the proof of Theorem B follows almost exactly the same lines as~\cite[Subsection 4.2]{Aus--thesis}; we include the following lemma again for completeness.

\begin{lem}[{C.f.~\cite[Proposition 4.2.6]{Aus--thesis}}]\label{lem:rel-ind}
If $\cal{I}$ and $\cal{J}$ are non-empty up-sets in $\binom{[d]}{\geq 2}$ and $\bfX$ is a fully sated $G^d$-space, then $\hat{\Phi}^{\cal{I}}_\bfX$ and $\hat{\Phi}^{\cal{J}}_\bfX$ are relatively independent over $\hat{\Phi}^{\cal{I}\cap \cal{J}}_\bfX$ under $\l$.
\end{lem}

\begin{proof}
\emph{Step 1.} \quad Suppose first that $\J = \langle e\rangle$ where $e$ is a maximal member of $\binom{[d]}{\geq 2}\setminus \I$.  Let $\{a_1,a_2,\ldots,a_m\}$ be the antichain of minimal elements of $\I$, so that $\hat{\Phi}^\I_\bfX = \bigvee_{k\leq m}\hat{\Phi}^{a_k}_\bfX$. The maximality assumption on $e$ implies that $e\cup \{j\}$ contains some $a_k$ for every $j \in [d]\setminus e$, and so $\I\cap \J$ is precisely the up-set generated by these sets $e\cup \{j\}$ for $j\in[d]\setminus e$.  We must therefore show that $\hat{\Phi}^e_\bfX$ is relatively independent from $\bigvee_{k\leq m}\hat{\Phi}^{a_k}_\bfX$ under $\l$ over the $\s$-subalgebra $\bigvee_{j\in [d]\setminus e}\hat{\Phi}^{e\cup\{j\}}_\bfX$.

Since $e \not\in \I$ we can find some $j_k \in
a_k\setminus e$ for each $k\leq m$.  Moreover, each $j\in
[d]\setminus e$ must appear as some $j_k$ in this list, since it
appears for any $k$ for which $a_k \subseteq e \cup \{j\}$.

Now Lemma~\ref{lem:diag} implies that $\hat{\Phi}^{a_k}_\bfX$ agrees with $\pi_{j_k}^{-1}(\Phi^{a_k}_\bfX)$ up to $\l$-negligible sets.  On the other hand, we clearly have $\pi_{j_k}^{-1}(\Phi^{a_k}_\bfX)\leq \pi_{j_k}^{-1}(\S_\bfX)$, and so in fact it will suffice to show that $\hat{\Phi}^e_\bfX$ is relatively independent from $\bigvee_{j\in [d]\setminus e}\pi_j^{-1}(\S_\bfX)$ over $\bigvee_{j\in [d]\setminus e}\hat{\Phi}^{e\cup\{j\}}_\bfX$.

Picking $i \in e$, Lemma~\ref{lem:diag} also gives that $\hat{\Phi}^e_\bfX$ agrees with $\pi_i^{-1}(\Phi^e_\bfX)$ up to $\l$-negligible sets.  On the other hand, for any sets
\[A_j \in \pi_j^{-1}(\S_\bfX)\ \hbox{for}\ j \in [d]\setminus e\ \hbox{and}\ B \in \Phi^e_\bfX,\]
the definition of $\l$ gives
\[\l\Big(\Big(\bigcap_{j \in [d]\setminus e}\pi_j^{-1}(A_j)\Big)\cap \pi_i^{-1}(B)\Big) = \lim_{n\to\infty}\int_X \L_n(f_1,\ldots,f_d)\,\d\mu\]
with
\[f_\ell := \left\{\begin{array}{ll}1_{A_\ell} & \quad \hbox{if}\ \ell \in [d]\setminus e\\ 1_B & \quad \hbox{if}\ \ell = i\\ 1 & \quad\hbox{else.} \end{array}\right.\]
By Proposition~\ref{prop:char-factors-again}, one obtains the same limit from
\[\lim_{n\to\infty}\int_X \L_n(f_1',\ldots,f_d')\,\d\mu\]
with
\[f'_\ell := \left\{\begin{array}{ll}1_{A_\ell} & \quad \hbox{if}\ \ell \in [d]\setminus e \\ \sfE_\mu(1_B\,|\,\Delta) & \quad \hbox{if}\ \ell = i\\ 1 & \quad\hbox{else,} \end{array}\right.\]
and where in turn
\[\Delta := \bigvee_{\ell \in [d]\setminus e,\,\ell < i}\S^{T_{(\ell;i]}}_\bfX \vee \bigvee_{\ell \in [d]\setminus e,\, \ell > i}\S^{T_{(i;\ell]}}_\bfX = \bigvee_{j \in [d]\setminus e}\Phi^{\{i,j\}}_\bfX.\]

This implies that $\pi_i^{-1}(\Phi^e_\bfX)$ is relatively independent from $\bigvee_{j\in [d]\setminus e}\pi_j^{-1}(\S_\bfX)$ over $\pi_i^{-1}(\Delta)$ under $\l$.  On the other hand, Corollary~\ref{cor:second-rel-ind} gives that $\Phi^e_\bfX$ is relatively independent from $\Delta$ over $\bigvee_{j\in [d]\setminus e}\Phi^{e\cup\{j\}}_\bfX$.  Combining these conclusions completes the proof in this case.

\vspace{7pt}

\emph{Step 2.}\quad The general case can now be treated for fixed
$\I$ by induction on $\J$. If $\J \subseteq \cal{I}$ then the
result is clear, so now let $e$ be a minimal member of
$\J\setminus\cal{I}$ of maximal size, and let $\cal{K} :=
\J \setminus\{e\}$.  It will suffice to prove that if $F \in
L^\infty(\l)$ is $\hat{\Phi}^\J_\bfX$-measurable,
then
\[\sfE_\l(F\,|\,\hat{\Phi}^\cal{I}_\bfX) = \sfE_\l(F\,|\,\hat{\Phi}^{\cal{I}\cap\J}_\bfX).\]
Furthermore, by an approximation in $\|\cdot\|_2$ by finite sums
of products, it suffices to prove this only for $F$ that are of the form $F_1\cdot
F_2$ with $F_1$ and $F_2$ being bounded and respectively
$\hat{\Phi}^{\langle e \rangle}_\bfX$- and
$\hat{\Phi}^{\cal{K}}_\bfX$-measurable. However, for such a product
we can write
\[\sfE_\l(F\,|\,\hat{\Phi}^{\cal{I}}_\bfX) =
\sfE_\l\big(\sfE_\l(F\,|\,\hat{\Phi}^{\cal{I}\cup\cal{K}}_\bfX)\,\big|\,\hat{\Phi}^{\cal{I}}_\bfX\big)
=
\sfE_\l\big(\sfE_\l(F_1\,|\,\hat{\Phi}^{\cal{I}\cup\cal{K}}_\bfX)\cdot
F_2\,\big|\,\hat{\Phi}^{\cal{I}}_\bfX\big).\]
By Step 1 we have
\[\sfE_\l(F_1\,|\,\hat{\Phi}^{\cal{I}\cup\cal{K}}_\bfX)
= \sfE_\l(F_1\,|\,\hat{\Phi}^{(\cal{I}\cup\cal{K})\cap\langle
e\rangle}_\bfX),\]
and on the other hand $(\cal{I}\cup\cal{K})\cap\langle
e\rangle \subseteq \cal{K}$ (because $\cal{K}$ contains every
subset of $[d]$ that strictly includes $e$, since $\J$ is an
up-set).  Therefore $(\cal{I}\cup\cal{K})\cap\langle
e\rangle = \K\cap \langle e\rangle$ and therefore another appeal to Step 1 gives
\[\sfE_\l(F_1\,|\,\hat{\Phi}^{(\cal{I}\cup\cal{K})\cap\langle
e\rangle}_\bfX) =
\sfE_\l(F_1\,|\,\hat{\Phi}^{\cal{K}}_\bfX).\]
Therefore the above expression for
$\sfE_\l(F_1F_2\,|\,\hat{\Phi}^{\cal{I}}_\bfX)$
simplifies to
\begin{multline*}
\sfE_\l\big(\sfE_\l(F_1\,|\,\hat{\Phi}^{\cal{K}}_\bfX)\cdot
F_2\,\big|\,\hat{\Phi}^{\cal{I}}_\bfX\big) =
\sfE_\l\big(\sfE_\l(F_1\cdot
F_2\,|\,\hat{\Phi}^{\cal{K}}_\bfX)\,\big|\,\hat{\Phi}^{\cal{I}}_\bfX\big)\\
= \sfE_\l\big(\sfE_\l(F\,|\,\hat{\Phi}^{\cal{K}}_\bfX)\,\big|\,\hat{\Phi}^{\cal{I}}_\bfX\big)
= \sfE_\l(F\,|\,\hat{\Phi}^{\cal{I}\cap
\cal{K}}_\bfX) =
\sfE_\l(F\,|\,\hat{\Phi}^{\cal{I}\cap\J}_\bfX),
\end{multline*}
where the third equality follows by the inductive hypothesis applied to $\cal{K}$ and $\cal{I}$.
\end{proof}

\begin{proof}[Proof of Theorem~\ref{thm:struct-of-coupling}]\emph{Part (1).}\quad Since $\bfX$ is fully sated, in particular it is sated with respect to the functorial $\s$-subalgebra
\[\bigvee_{\ell=1}^{i-1}\S^{T_{(\ell;i]}}_\bfX \vee \bigvee_{\ell=i+1}^d\S^{T_{(i;\ell]}}_\bfX = \bigvee_{\ell=1}^{i-1}\Phi^{\{\ell,i\}}_\bfX \vee \bigvee_{\ell=i+1}^d\Phi^{\{i,\ell\}}_\bfX = \Phi_\bfX^{\langle i\rangle}\]
for each $1 \leq i \leq d$.   Therefore Proposition~\ref{prop:char-factors-again} gives
\begin{eqnarray*}
\int_{X^d} f_1\otimes \cdots \otimes f_d\,\d\l &=& \lim_{n\to\infty} \frac{1}{|F_n|}\sum_{g \in F_n} \int_X \prod_{i=1}^d (f_i\circ T_{[1;i]}^g)\,\d\mu\\
&=& \lim_{n\to\infty} \frac{1}{|F_n|}\sum_{g \in F_n} \int_X \prod_{i=1}^d(\sfE_\mu(f_i\,|\,\Phi^{\langle i\rangle}_\bfX)\circ T_{[1;i]}^g)\,\d\mu\\
 &=& \int_{X^d} \sfE_\mu(f_1\,|\,\Phi^{\langle 1\rangle}_\bfX)\otimes \cdots \otimes \sfE_\mu(f_d\,|\,\Phi^{\langle d\rangle}_\bfX) \,\d\l
\end{eqnarray*}
for any $f_1$, \ldots, $f_d \in L^\infty(\mu)$.

\vspace{7pt}

\emph{Part (2).}\quad Property (i) of Proposition~\ref{prop:infremove} holds by construction of the $\s$-algebras $\Phi^e_\bfX$; property (ii) is given by Lemma~\ref{lem:diag}; and property (iii) is given by Lemma~\ref{lem:rel-ind}.
\end{proof}

\bibliographystyle{alpha}
\bibliography{bibfile}

\def\cprime{$'$}
\begin{thebibliography}{Aus10b}

\bibitem[Aus]{Aus--lindeppleasant1}
Tim Austin.
\newblock Pleasant extensions retaining algebraic structure, {I}.
\newblock Preprint, available online at \verb|arXiv.org|: 0905.0518.

\bibitem[Aus09]{Aus--nonconv}
Tim Austin.
\newblock On the norm convergence of nonconventional ergodic averages.
\newblock {\em Ergodic Theory Dynam. Systems}, 30(2):321--338, 2009.

\bibitem[Aus10a]{Aus--newmultiSzem}
Tim Austin.
\newblock Deducing the multidimensional {S}zemer\'edi theorem from an
  infinitary removal lemma.
\newblock {\em J. d'Analyse Math.}, 111:131--150, 2010.

\bibitem[Aus10b]{Aus--thesis}
Timothy~Derek Austin.
\newblock {\em Multiple recurrence and the structure of probability-preserving
  systems}.
\newblock ProQuest LLC, Ann Arbor, MI, 2010.
\newblock Thesis (Ph.D.)--University of California, Los Angeles.

\bibitem[BH92]{BerHin92}
Vitaly Bergelson and Neil Hindman.
\newblock Some topological semicommutative van der {W}aerden type theorems and
  their combinatorial consequences.
\newblock {\em J. London Math. Soc. (2)}, 45(3):385--403, 1992.

\bibitem[BM07]{BerMcC07}
V.~Bergelson and R.~McCutcheon.
\newblock Central sets and a non-commutative {R}oth theorem.
\newblock {\em Amer. J. Math.}, 129(5):1251--1275, 2007.

\bibitem[BMZ97]{BerMcCZha97}
Vitaly Bergelson, Randall McCutcheon, and Qing Zhang.
\newblock A {R}oth theorem for amenable groups.
\newblock {\em Amer. J. Math.}, 119(6):1173--1211, 1997.

\bibitem[FK78]{FurKat78}
Hillel Furstenberg and Yitzhak Katznelson.
\newblock An ergodic {S}zemer\'edi {T}heorem for commuting transformations.
\newblock {\em J. d'Analyse Math.}, 34:275--291, 1978.

\bibitem[Gla03]{Gla03}
Eli Glasner.
\newblock {\em Ergodic {T}heory via {J}oinings}.
\newblock American {M}athematical {S}ociety, {P}rovidence, 2003.

\bibitem[HK05]{HosKra05}
Bernard Host and Bryna Kra.
\newblock Nonconventional ergodic averages and nilmanifolds.
\newblock {\em Ann. Math.}, 161(1):397--488, 2005.

\bibitem[Hos09]{Hos09}
Bernard Host.
\newblock Ergodic seminorms for commuting transformations and applications.
\newblock {\em Studia Math.}, 195(1):31--49, 2009.

\bibitem[Tao06]{Tao06--hyperreg}
Terence Tao.
\newblock Szemer\'edi's regularity lemma revisited.
\newblock {\em Contrib. Discrete Math.}, 1(1):8--28, 2006.

\bibitem[Tao07]{Tao07}
Terence Tao.
\newblock A correspondence principle between (hyper)graph theory and
  probability theory, and the (hyper)graph removal lemma.
\newblock {\em J. d'Analyse Math.}, 103:1--45, 2007.

\bibitem[Wal12]{Walsh12}
Miguel~N. Walsh.
\newblock Norm convergence of nilpotent ergodic averages.
\newblock {\em Ann. of Math. (2)}, 175(3):1667--1688, 2012.

\bibitem[ZK]{Zor13}
Pavel Zorin-Kranich.
\newblock Norm convergence of multiple ergodic averages on amenable groups.
\newblock To appear, \emph{J. d'Analyse Math.}

\end{thebibliography}

\end{document}